\documentclass[11pt]{article}
\usepackage{times}
\usepackage{amsopn,amssymb,amsthm,amsmath}
\usepackage{paralist}
\usepackage{algorithm}
\usepackage{algorithmic}
\usepackage{color}
\usepackage{graphicx}
\usepackage{comment}
\usepackage{hyperref}
\usepackage{multirow}
\usepackage{float}
\usepackage{subfigure}
\usepackage{epstopdf}
\usepackage{natbib}
\usepackage{xspace}

\usepackage{tikz}
\usetikzlibrary{shapes.gates.logic.US,trees,positioning,arrows}
\usepackage{tablefootnote}

\newlength{\defbaselineskip}
\newtheorem{theorem}{Theorem}

\newtheorem{definition}[theorem]{Definition}
\newtheorem{lemma}[theorem]{Lemma}
\newtheorem{proposition}[theorem]{Proposition}
\newtheorem{corollary}[theorem]{Corollary}

\def\X{\mathcal{X}}
\def\Y{\mathcal{Y}}

\def\F{\mathcal{F}}

\def\Z{\mathcal{Z}}
\def\A{\mathcal{A}}
\def\C{\mathcal{C}}
\def\D{\mathcal{D}}

\DeclareMathOperator{\poly}{poly}
\DeclareMathOperator{\bigO}{\mathcal{O}}

\DeclareMathOperator{\nnz}{nnz}
\DeclareMathOperator{\argmin}{argmin}

\newcommand{\ExpectSub}[2]{\mbox{}{\mathbb{E}}_{#1}\left[#2\right]}
\newcommand{\Expect}[1]{\mbox{}{\mathbb{E}}\left[#1\right]}

\def\reals{\mathbb{R}}

\renewcommand{\algorithmicrequire}{\textbf{Input:}}
\renewcommand{\algorithmicensure}{\textbf{Output:}}

\newcommand{\pcsgd}{\textsc{pwSGD}\xspace}

\begin{document}

\title{\Large Weighted SGD for $\ell_p$ Regression with Randomized Preconditioning
\footnote{A conference version of this paper appears under
the same title in {\it Proceedings of ACM-SIAM Symposium on Discrete Algorithms}, Arlington, VA, 2016~\citep{YCRM16_SODA}.}
}

\author{
  Jiyan Yang
  \thanks{
    Institute for Computational and Mathematical Engineering,
    Stanford University,
    Email: \{jiyan, ychow\}@stanford.edu
  }
\and
 Yin-Lam Chow
  \footnotemark[1]
\and
 Christopher R{\'e}
 \thanks{
    Department of Computer Science,
    Stanford University,
    Email: chrismre@cs.stanford.edu
  }
\and
Michael W. Mahoney
\thanks{
 International Computer Science Institute 
and Department of Statistics,
University of California, Berkeley,
mmahoney@stat.berkeley.edu
}
}

\date{}
\maketitle

\begin{abstract} 
In recent years, 
stochastic gradient descent (SGD) methods and randomized linear algebra 
(RLA) algorithms have been applied to many large-scale problems in 
 machine learning and data analysis. 
SGD methods are easy to implement and applicable to a wide range of convex optimization problems. 
In contrast, RLA algorithms provide much stronger performance 
guarantees but are applicable to a narrower class of problems. We aim to bridge the gap between these two methods in solving {\it constrained} overdetermined linear regression problems---e.g., $\ell_2$ and $\ell_1$ regression problems.
\begin{compactitem}
\item
We propose a hybrid algorithm named \pcsgd that uses RLA techniques for preconditioning and constructing an importance sampling distribution, and then performs an SGD-like iterative process with weighted sampling on the preconditioned system. 
\item
By rewriting a deterministic $\ell_p$ regression problem as a stochastic optimization problem, we connect \pcsgd to several existing $\ell_p$ solvers including RLA methods with algorithmic leveraging (RLA for short).
\item
We prove that \pcsgd inherits faster convergence rates that only depend on the lower dimension of the linear system, while maintaining low computation complexity.
Such SGD convergence rates are superior to other related SGD algorithm such as the weighted randomized Kaczmarz algorithm.
\item
Particularly, when solving $\ell_1$ regression with size $n$ by $d$, \pcsgd returns an approximate solution with $\epsilon$ relative error in the objective value in $\bigO(\log n \cdot \nnz(A) + \poly(d)/\epsilon^2)$ time.
This complexity is {\it uniformly} better than that of RLA methods in terms of both $\epsilon$ and $d$ when the problem is unconstrained. 
In the presence of constraints, \pcsgd only has to solve a sequence of much simpler and smaller optimization problem over the same constraints. In general this is more efficient than solving the constrained subproblem required in RLA.
\item
For $\ell_2$ regression, \pcsgd returns an approximate solution with $\epsilon$ relative error in the objective value and the solution vector measured in prediction norm in $\bigO(\log n \cdot \nnz(A) + \poly(d) \log(1/\epsilon) /\epsilon)$ time.
We show that for unconstrained $\ell_2$ regression, this complexity is comparable to that of RLA and is asymptotically better over several state-of-the-art solvers in the regime where the desired accuracy $\epsilon$, high dimension $n$ and low dimension $d$ satisfy $d\geq 1/\epsilon$ and $n \geq d^2/\epsilon$.
\end{compactitem}
We also provide lower bounds on the coreset complexity for more general regression problems, indicating that still new ideas will be needed to extend similar RLA preconditioning ideas to weighted SGD algorithms for more general regression problems.
Finally, the effectiveness of such algorithms is illustrated numerically on both synthetic and real datasets, and the results are consistent with our theoretical findings and demonstrate that \pcsgd converges to a medium-precision solution, e.g., $\epsilon=10^{-3}$, more quickly.
\end{abstract}

\section{Introduction}
\label{sec:intro}
Many novel algorithms for large-scale data analysis and machine learning problems have emerged in recent years, among which stochastic gradient descent 
(SGD) methods and randomized linear algebra (RLA) algorithms have received much
attention---both for their strong performance in practical applications and 
for their interesting theoretical 
properties~\citep{bottou-2010,Mah-mat-rev_BOOK}. 
Here, we consider the ubiquitous $\ell_1$ and $\ell_2$ regression problems, and we describe a novel RLA-SGD algorithm called \pcsgd (preconditioned weighted SDG).  Our new algorithm combines the advantages of both RLA and SGD methods for 
solving constrained overdetermined $\ell_1$ and $\ell_2$ 
regression problems.

Consider the overdetermined $\ell_p$ regression problem
\begin{equation}
  \label{eq:lp_obj}
     \min_{x \in \Z} f(x) = \|A x-b\|_p,
 \end{equation}
where $p\in[1,\infty]$, $A \in \reals^{n\times d}$, $b \in \reals^n$ and $n\gg d$.
When $\Z = \reals^d$, i.e., the solution space is unconstrained, 
the cases $p\in\{1,2\}$ are respectively known as the Least Absolute Deviations 
(LAD, or $\ell_1$) and Least-squares (LS, or $\ell_2$) regression 
problems.
Classically, the unconstrained $\ell_2$ regression problem can be solved by 
eigenvector-based methods with worst-case 
running time $\bigO(nd^2)$~\citep{GVL96}; or by 
iterative methods for which the running time 
depends on the condition number of $A$~\citep{templates,Kelly95,Saad03},
 while
the unconstrained $\ell_1$ regression problem can be formulated as a linear 
program~\citep{PK97,BP} and solved by an interior-point 
method~\citep{PK97,Por97}.

For these and other regression problems,
SGD algorithms are widely used in practice because of their scalability and efficiency.
In contrast, RLA algorithms have better theoretical guarantees but (thus far) have been less flexible, e.g., in the presence of constraints.
For example, they may use an interior point method for solving a constrained subproblem, and this may be less efficient than SGD. 
(Without constraints, RLA methods can be used to construct subproblems to be solved exactly, or they can be used to construct preconditioners for the original problem; see~\citet{yang16pardist} for details and implementations of these RLA methods to compute low, medium, and high precision solutions on up to terabyte-sized input data.)
In this paper, we combine these two algorithmic approaches to develop a method that takes advantage of the strengths of both of these approaches.

\subsection{Overview of our main algorithm}
Our main algorithm \pcsgd is a hybrid method for solving constrained overdetermined $\ell_1$ and $\ell_2$ regression problems. It consists of two main steps. First, apply RLA techniques for preconditioning 
and construct an importance sampling distribution. Second, apply an SGD-like iterative phase with weighted sampling on the preconditioned system.
Such an algorithm preserves the simplicity of SGD and the high quality theoretical guarantees of RLA.
In particular, we prove that after preconditioning, the number of iterations required to converge to a target accuracy is fully predictable and only depends on the low dimension $d$, i.e., it is independent of the high dimension $n$.
We show that, with a proper choice of preconditioner,
\pcsgd runs in $\bigO( \log n \cdot \nnz(A) + \poly(d)/\epsilon^2)$ time to return an approximate solution with $\epsilon$ relative error in the objective for constrained $\ell_1$ regression; and in $\bigO( \log n \cdot \nnz(A) + \poly(d)\log(1/\epsilon)/\epsilon)$ time to return an approximate solution with $\epsilon$ relative error in the solution vector in prediction norm for constrained $\ell_2$ regression. Furthermore, for unconstrained $\ell_2$ regression, \pcsgd runs in $\bigO( \log n \cdot \nnz(A) + d^3\log(1/\epsilon)/\epsilon)$ time to return an approximate solution with $\epsilon$ relative error in the objective.

To provide a quick overview of how \pcsgd compares to existing algorithms, in Tables~\ref{table:complexity_l1}~and~\ref{table:complexity}, we summarize the complexity required to compute a solution $\hat x$ with relative error $(f(\hat x) - f(x^\ast))/f(x^\ast) = \epsilon$, of several solvers for unconstrained $\ell_1$ and $\ell_2$ regression.
In Table~\ref{table:complexity_l1}, RLA with algorithmic leveraging (RLA for short)~\citep{CDMMMW12,YMM14_SISC} is a popular method for obtaining a low-precision solution and randomized IPCPM is an iterative method for finding a higher-precision solution~\citep{MM13_RRinMR} for unconstrained $\ell_1$ regression. Clearly, \pcsgd has a uniformly better complexity than that of RLA methods in terms of both $d$ and $\epsilon$, no matter which underlying preconditioning method is used. This makes \pcsgd a more suitable candidate for getting a medium-precision, e.g., $\epsilon = 10^{-3}$, solution. 

In Table~\ref{table:complexity},
all the methods require constructing a sketch first. Among them,
 ``low-precision'' solvers refer to ``sketching + direct solver'' type algorithms; see \citep{drineas2011faster, CW12} for projection-based examples and \citep{CW12, DMMW12} for sampling-based examples.
``High-precision'' solvers refer to ``sketching + preconditioning + iterative solver'' type algorithms; see \citep{BLENDENPIK, MSM14_SISC} for examples.
One can show that, when $d\geq 1/\epsilon$ and $n \geq d^2/\epsilon$, \pcsgd is asymptotically better than all the solvers shown in Table~\ref{table:complexity}.
Moreover, although high-precision solvers are more efficient when a high-precision solution is desired, usually they are designed for unconstrained problems, whereas \pcsgd also works for constrained problems.

{\it We remark that, compared to general SGD algorithms, our RLA-SGD hybrid algorithm \pcsgd works for problems in a narrower range, i.e., $\ell_p$ regression, but inherits the strong theoretical guarantees of RLA. When solving $\ell_2$ regression, for which traditional RLA methods are well designed, \pcsgd has a comparable complexity. On the other hand, when solving $\ell_1$ regression, due to the efficiency of SGD update, \pcsgd has a strong advantage over traditional RLA methods.}
See Sections~\ref{sxn:complexity}~and~\ref{sec:rla} for more detailed discussions.

Finally, in Section~\ref{sxn:experiments}, empirically we show that \pcsgd performs favorably compared to other competing methods, as it converges to a medium-precision solution more quickly.

\begin{table}[H]
  \centering
  \small
   \begin{tabular}{c|cc}
     solver  & complexity (general) & complexity (sparse)  \\
    \hline
  RLA with algorithmic leveraging & $time(R) + \bigO(\nnz(A) \log n  + \bar \kappa_1^{\frac{5}{4}} d^{\frac{17}{4}} / \epsilon^{\frac{5}{2}})$ & $\bigO(\nnz(A) \log n  +d^{\frac{69}{8}} \log^{\frac{25}{8}}d / \epsilon^{\frac{5}{2}})$\\
  randomized IPCPM & $time(R) + nd^2 + \bigO( (nd+\poly(d))\log(\bar \kappa_1 d/\epsilon) )$ & $nd^2 + \bigO((nd+\poly(d)) \log(d/\epsilon))$ \\
  \pcsgd & $time(R) + \bigO(\nnz(A) \log n + d^3 \bar \kappa_1/\epsilon^2) $ & 
  $\bigO(\nnz(A) \log n + d^{\frac{13}{2}} \log^{\frac{5}{2}}d / \epsilon^2)$
  \end{tabular}
   \caption{
    \small
   Summary of complexity of several unconstrained $\ell_1$ solvers that use randomized linear algebra. 
   The target is to find a solution $\hat x$ with accuracy $(f(\hat x) - f(x^\ast)) /f(x^\ast) \leq \epsilon$, where $f(x) = \|Ax-b\|_1$.
   In the above, $time(R)$ denotes the time needed to compute a matrix $R$ such that $AR^{-1}$ is well-conditioned with condition number $\bar \kappa_1$ (Definition~\ref{def:basis}).
   The general complexity bound and the one using sparse reciprocal exponential transform~\citep{WZ13} as the underlying sketching method are presented. Here, we assume $n\gg d$ such that $n > d^3\log d$ and the underlying $\ell_1$ regression solver in RLA with algorithmic leveraging algorithm takes $\bigO(n^{\frac{5}{4}}d^3)$ time to return a solution~\citep{PK97}. The complexity of each algorithm is computed by setting the failure probability to be a constant. }
   \label{table:complexity_l1}
 \end{table}

\begin{table}[H]
  \centering
  \small
   \begin{tabular}{c|cc}
     solver  & complexity (SRHT) & complexity (CW)  \\
    \hline
 low-precision solvers (projection) & $\bigO\left(nd \log (d/\epsilon) + d^3 \log n(\log d + 1/\epsilon)\right)$ & $\bigO\left(\nnz(A) + d^4/\epsilon^2 \right)$  \\
 low-precision solvers (sampling) & $\bigO\left(nd \log n + d^3 \log n \log d + d^3 \log d/\epsilon\right)$ & $\bigO\left(\nnz(A)\log n + d^4 + d^3 \log d/\epsilon \right)$ \\
 high-precision solvers & $\bigO \left(nd \log n + d^3 \log n \log d + nd \log(1/\epsilon) \right)$ 
  & $\bigO\left( \nnz(A) + d^4 + nd \log(1/\epsilon) \right)$ \\
  \pcsgd & $\bigO\left( nd \log n + d^3 \log n \log d + d^3\log(1/\epsilon) /\epsilon\right)$ & $\bigO\left( \nnz(A) \log n + d^4 + d^3\log(1/\epsilon)/\epsilon\right)$
  \end{tabular}
   \caption{
   \small
   Summary of complexity of several unconstrained $\ell_2$ solvers that use randomized linear algebra. 
   The target is to find a solution $\hat x$ with accuracy $(f(\hat x) - f(x^\ast)) /f(x^\ast) \leq \epsilon$, where $f(x) = \|Ax-b\|_2$. Two sketching methods, namely, SRHT~\citep{drineas2011faster,tropp2011improved} and CW~\citep{CW12} are considered. Here, we assume $d \leq n \leq e^d$. The complexity of each algorithm is computed by setting the failure probability to be a constant.}
    \label{table:complexity}
 \end{table}

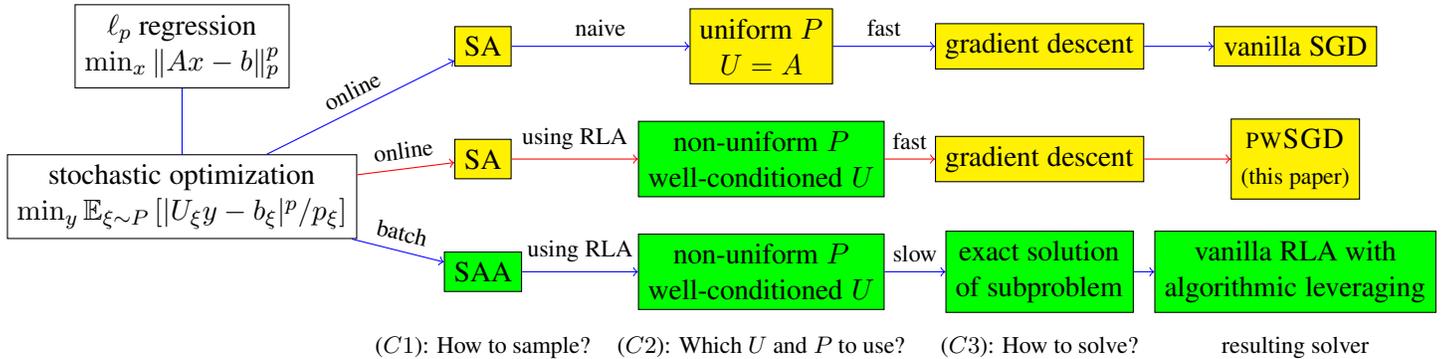
\begin{figure*}[h!tpb]
\hspace{-0.5in}
\begin{tikzpicture}
  \node[draw,align=center] (lp) at (-0.8,2) {$\ell_p$ regression \\
     $ \min_x \|Ax-b\|_p^p$};
  \node[draw, align = center] (sto) at (-0.8,0) {stochastic optimization\\ $\min_y \ExpectSub{\xi \sim P}{\lvert U_\xi y - b_\xi \rvert^p/p_\xi}$};
  
  \draw[-,draw=blue] (lp) -- (sto);
    
  \node[draw,align=center,fill=yellow] (q1_sgd) at (3.2,2) {SA};
   \node[draw,align=center,fill=yellow] (q1_pcsgd) at (3.2,0.5) {SA};
  \node[draw,align=center,fill=green] (q1_rla) at (3.2,-1) {SAA};
    
  \draw[->,draw=blue] (sto) -- (q1_sgd) node [midway, above, sloped] {{\footnotesize online}};
  \draw[->,draw=red] (sto) -- (q1_pcsgd) node [midway, above, sloped] {{\footnotesize online}};
  \draw[->,draw=blue] (sto) -- (q1_rla) node [midway, above, sloped] {{\footnotesize batch}};
  
  \node[align=center] (q1) at (3.2,-2) {{\footnotesize ($C1$): How to sample?}};
  
  \node[draw,align=center,fill=yellow] (q2_sgd) at (6.9,2) {uniform $P$\\
  $U=A$};
  \node[draw,align=center,fill=green] (q2_pcsgd) at (6.9,0.5) {non-uniform $P$\\ well-conditioned $U$};
  \node[draw,align=center,fill=green] (q2_rla) at (6.9,-1) {non-uniform $P$\\ well-conditioned $U$};
  
  \draw[->,draw=blue] (q1_sgd) -- (q2_sgd) node [midway, above, sloped] {{\footnotesize naive}};
  \draw[->,draw=red] (q1_pcsgd) -- (q2_pcsgd) node [midway, above, sloped] {{\footnotesize using RLA}};
  \draw[->,draw=blue] (q1_rla) -- (q2_rla) node [midway, above, sloped] {{\footnotesize using RLA}};
  
  \node[align = center] (q2) at (6.9,-2) {{\footnotesize ($C2$): Which $U$ and $P$ to use?}};
  
  \node[draw,fill=yellow] (q3_sgd) at (10.6,2) {gradient descent};
  \node[draw,fill=yellow] (q3_pcsgd) at (10.6,0.5) {gradient descent};
  \node[draw,fill=green,align=center] (q3_rla) at (10.6,-1) {exact solution \\ of subproblem};
  
  \draw[->,draw=blue] (q2_sgd) -- (q3_sgd) node [midway, above, sloped] {{\footnotesize fast}};
  \draw[->,draw=red] (q2_pcsgd) -- (q3_pcsgd) node [midway, above, sloped] {{\footnotesize fast}};
  \draw[->,draw=blue] (q2_rla) -- (q3_rla) node [midway, above, sloped] {{\footnotesize slow}};
  
   \node[align = center] (q3) at (10.6,-2) {{\footnotesize ($C3$): How to solve?}};
   
   \node[draw,fill=yellow] (sgd) at (14,2) {vanilla SGD};
  \node[draw,fill=yellow,align=center] (pcsgd) at (14,0.5) {\pcsgd \\ {\footnotesize (this paper)}};
  \node[draw,fill=green,align=center] (rla) at (14,-1) {vanilla RLA with \\ algorithmic leveraging};
  
  \draw[->,draw=blue] (q3_sgd) -- (sgd);
  \draw[->,draw=red] (q3_pcsgd) -- (pcsgd);
  \draw[->,draw=blue] (q3_rla) -- (rla);
  
   \node[align = center] (q3_text) at (14,-2) {{\footnotesize resulting solver}};

\end{tikzpicture}
\caption{An overview of our framework for solving $\ell_p$ regression via stochastic optimization. 
To construct a solver, three choices have to be made.
For ($C1$), the answer can be either SAA (Sampling Average Approximation, i.e., sample a batch of points and deal with the subproblem) or SA (Stochastic Approximation, i.e., sample a mini-batch in an online fashion and update the weight vector after extracting useful information).
In ($C2$), the answer is determined by $P$, which denotes the underlying sampling distribution (uniform or nonuniform) and $U$, which denotes the basis with which to work (original or preconditioned system).
Finally, for ($C3$), the answer determines how we solve the subproblem (in SAA) or what information we extract and how we update the weight (in SA).}
\label{fig:flow}
\end{figure*}

\subsection{Connection to related algorithms}
As a side point of potentially independent interest, a connection between $\ell_p$ regression and stochastic optimization will allow us to unify our main algorithm \pcsgd and some existing $\ell_p$ regression solvers under the same framework.
In Figure~\ref{fig:flow}, we present the basic structure of this framework, which provides a view of \pcsgd from another perspective.
To be more specific, we (in Proposition~\ref{prop:det-lp-to-stoch-optiz} formally) reformulate the deterministic $\ell_p$  regression problem in \eqref{eq:lp_obj}
as a stochastic optimization problem, i.e.,
\begin{equation*}
  \min_{y \in \Y} \|Uy-b\|_p^p = \min_{y \in \Y} \ExpectSub{\xi \sim P}{\lvert U_\xi y - b_\xi \rvert^p/p_\xi},
\end{equation*}
 where
$U$ is a basis for the range space of $A$ and
$\xi$ is a random variable over 
$\{1,\ldots,n\}$ with distribution $P=\{p_i\}_{i=1}^n$.
As suggested in Figure~\ref{fig:flow}, to solve this stochastic optimization problem,
typically one needs to answer the following three questions.
\begin{compactitem}
\item {\it ($C1$): How to sample: SAA (Sampling Average Approximation, i.e., draw samples in a batch mode and deal with the subproblem) or SA (Stochastic Approximation, i.e., draw a mini-batch of samples in an online fashion and update the weight after extracting useful information)?}
\item {\it ($C2$): Which probability distribution $P$ (uniform distribution or not) and which basis $U$ (preconditioning or not) to use?}
\item {\it ($C3$): Which solver to use (e.g., how to solve the subproblem in SAA or how to update the weight in~SA)?}
\end{compactitem}
Some combinations of these choices may lead to existing solvers; see Figure~\ref{fig:flow} and Section~\ref{sxn:naive_alg} for more details.
A natural question arises: is there a combination of these choices that leverages the algorithmic benefits of RLA preconditioning to improve the performance of SGD-type algorithms? 
Recall that RLA methods (in particular, those that exploit algorithmic averaging; see Appendix~\ref{sec:related_alg} and also \citep{DMMW12,yang16pardist}) inherit strong theoretical guarantees because the underlying sampling distribution $P$ captures most of the important information of the original system; moreover, such a carefully constructed leverage-based distribution is defined based on a well-conditioned basis $U$, e.g., an orthogonal matrix for $p=2$.
One immediate idea is to develop an SGD-like algorithm that uses the same choice of $U$ and $P$ as in RLA methods.
This simple idea leads to our main algorithm \pcsgd, which is an online algorithm ($C1$) that uses a non-uniform sampling distribution ($C2$) and performs a gradient descent update ($C3$) on a preconditioned system ($C2$), as Figure~\ref{fig:flow} suggests.

Indeed, for least-squares problems (unconstrained $\ell_2$ regression), \pcsgd is highly related to the weighted randomized Kaczmarz (RK) algorithm~\citep{SV09, needel-weightedsgd} in the way that both algorithms are SGD algorithm with non-uniform $P$ but \pcsgd runs on a well-conditioned basis $U$ while randomized RK doesn't involve preconditioning. 
In Section~\ref{sec:connection_full} we show that this preconditioning step dramatically reduces the number of iterations required for \pcsgd to converge to a (fixed) desired accuracy.

\subsection{Main contributions}

Now we are ready to state our main contributions.

\begin{compactitem}
 \item We reformulate the deterministic $\ell_p$ regression problem~\eqref{eq:lp_obj} into a stochastic optimization problem~\eqref{eq:lp_sto} and make connections to existing solvers including RLA methods with algorithmic leveraging and weighted randomized Kaczmarz algorithm (Sections~\ref{sxn:naive_alg}~and~\ref{sec:connection_full}).
 \item We develop a hybrid algorithm for solving {\it constrained} overdetermined $\ell_1$ and $\ell_2$ regression called \pcsgd, which is an SGD algorithm with preconditioning and a non-uniform sampling distribution constructed using RLA techniques.
We present several choices of the preconditioner and their tradeoffs.
We show that with a suitable preconditioner, convergence rate of the SGD phase only depends on the low dimension $d$, and is independent of the high dimension $n$ (Sections~\ref{sec:performance_alg}~and~\ref{sxn:F}). 
\item
We prove that \pcsgd returns an approximate solution with $\epsilon$ relative error in the objective value in $\bigO(\log n \cdot \nnz(A) + \poly(d)/\epsilon^2)$ time for $\ell_1$ regression. This complexity is {\it uniformly} better than that of RLA methods in terms of both $\epsilon$ and $d$ when the problem is unconstrained. 
In the presence of constraints, \pcsgd only has to solve a sequence of much simpler and smaller optimization problems over the same constraints, which in general can be more efficient than solving the constrained subproblem required in RLA (Sections~\ref{sxn:complexity}~and~\ref{sec:rla}).
\item
We prove that \pcsgd returns an approximate solution with $\epsilon$ relative error in the objective value and the solution vector measured in prediction norm in $\bigO(\log n \cdot \nnz(A) + \poly(d) \log(1/\epsilon) /\epsilon)$ time for $\ell_2$ regression.
We show that for unconstrained $\ell_2$ regression, this complexity is asymptotically better than several state-of-the-art solvers in the regime where $d\geq 1/\epsilon$ and $n \geq d^2/\epsilon$ (Sections~\ref{sxn:complexity}~and~\ref{sec:rla}).
\item
Empirically, we show that when solving $\ell_1$ and $\ell_2$ regression problems, \pcsgd inherits faster convergence rates and performs favorably in the sense that it obtains a medium-precision much faster than other competing SGD-like solvers do.
Also, theories regarding several choices of preconditioners are numerically verified (Section~\ref{sxn:experiments}).
\item
We show connections between RLA algorithms and coreset methods of empirical optimization problems under the framework of~\citet{Feldman_coreset}.
We show that they are equivalent for $\ell_p$ regression and provide lower bounds on the 
coreset complexity for some more general regression problems. 
We also discuss the difficulties in extending similarly RLA preconditioning ideas to general SGD algorithms (Section~\ref{sec:coreset}).
\end{compactitem}



\vspace{-2mm}
\subsection{Other prior related work}
Numerous RLA algorithms have been proposed to solve $\ell_p$ regression problems~\citep{yang16pardist}.
RLA theories show that to achieve a relative-error bound, the 
required sampling size only depends on $d$, independent of $n$, and the 
running time also depends on the time to implement the random projection at the 
first step.
Regarding the performance of unconstrained regression problems, in \citep{DDHKM09} the authors provide an algorithm that constructs a well-conditioned basis by ellipsoid rounding and a subspace-preserving sampling matrix for $\ell_p$ regression problems
in $\bigO(nd^5 \log n)$ time; 
a sampling algorithm based on Lewis weights for $\ell_p$ regression have been proposed by~\cite{cohen15lewis};
the algorithms in~\citep{SW11} and~\citep{CDMMMW12} use the ``slow'' and 
``fast''  Cauchy Transform to compute the low-distortion $\ell_1$ 
embedding matrix and solve the over-constrained $\ell_1$ regression problem 
in $\bigO(nd^{1.376+})$ and $\bigO(nd\log n)$ time, respectively; 
the algorithms in~\citep{DMMW12} estimate the leverage scores up to a small factor and solve the $\ell_2$ regression problem in $\bigO(nd\log n)$ time respectively; and
the algorithms in~\citep{CW12,MM12,nelson2013sparse}, solve the problem via sparse random projections in 
nearly input-sparsity time, i.e., $\bigO(\log n \cdot \nnz(A))$ time, plus lower-order terms, and a tighter analysis is provided by \citet{cohen2016nearly}.
As for iterative algorithms, the algorithms in
\citep{BLENDENPIK, MSM14_SISC} use randomized linear algebra to compute a preconditioner and call iterative solvers such as LSQR to solve the preconditioned problem.


In contrast, SGD algorithms update the solution vector in an iterative fashion
and are simple to implement and scalable to large datasets  \citep{bottouLargeScale,svm_opt,bottou2008tradeoffs}.
Moreover, these methods can be easily extended for problems with general
convex loss functions and constraints, such as Pegasos~\citep{pegasos} for regularized SVM and 
stochastic coordinate descent (SCD) for $\ell_1$ regularization~\citep{l1_reg}.
Several techniques, such as 
SAGE~\citep{accelerate_sgd}, 
AdaGrad~\citep{adagrad}, SVRG~\citep{variance_sgd},
 have recently been proposed to accelerate the 
convergence rate of SGD, and~\citet{hogwild} also show that SGD is favorable for parallel/distributed computation.
More recently, several works, e.g.,~\citep{ZZ15, needel-weightedsgd} regarding SGD with weighted sampling are proposed, in which the authors show that the performance of SGD can be improved by using a nonuniform sampling~distribution.

In addition, as we point out in Section~\ref{sxn:F}, \pcsgd has a close relationship to second-order methods. It can be viewed as an algorithm with approximate Hessians obtained by sketching and stochastic gradients. 
This is related to the iterative Hessian sketching algorithm for solving constrained least squares problems proposed by~\citet{pilanci2014iterative} which is essentially a Newton-type algorithm with iterative sketched Hessians and batch gradients. Moreover, the idea of using approximate Hessians and stochastic gradients have been discussed in several recent papers. For example, \citep{Moritz16, Byrd16, Curtis16} exploit the idea of approximating Hessian with L-BFGS type updates and (variance-reduced) stochastic updates.

\section{Preliminaries}


For any matrix $A \in \reals^{n \times d}$, we use $A_i$ and $A^j$ to denote 
the $i$-th row and $j$-th column of $A$, respectively. 
We assume $A$ has full rank, i.e., $\text{rank}(A) = d$.
Also denote by $\kappa(A)$ the usual condition number of $A$, by $\nnz(A)$ the number of nonzero elements in $A$, and by $\poly(d)$ a low-degree polynomial in $d$.
We also use $[n]$ to denote the set of indices $1,\ldots,n$.

Throughout this subsection, the definitions are applied 
to general $p \in [1,\infty)$.
We denote by $| \cdot |_p$ the element-wise $\ell_p$ norm of a matrix:
$|A|_p = \left(\sum_{i=1}^n \sum_{j=1}^d |A_{ij}|^p \right)^{1/p}$.
In particular, when $p = 2$, $|\cdot|_2$ is equivalent to the 
Frobenius norm.

The following two notions on well-conditioned bases and leverage scores are crucial to our methods.
The first notion is originally introduced 
by \cite{Cla05} and stated more precisely in \cite{DDHKM09}, and it is used to justify the well-posedness of a $\ell_p$ regression problem.
These notions were introduced by \citet{DDHKM09}.
\begin{definition}[$(\alpha, \beta, p)$-conditioning and well-conditioned basis]
\label{def:basis}
An $A \in \reals^{n \times d}$ is $(\alpha, \beta, p)$-conditioned 
if $ |A|_p \leq \alpha$ and for all $x\in \reals^d$, 
$\beta \|Ax\|_p \geq \|x\|_q$, where $1/p + 1/q = 1$. 
Define $\bar \kappa_p(A)$ as the minimum value of $\alpha \beta$ such that $A$ is 
$(\alpha, \beta,p)$-conditioned.
We say that a basis $U$ for $\textup{range}(A)$ is a well-conditioned 
basis if $\bar \kappa_p=\bar \kappa_p(U)$ is a low-degree polynomial in $d$, independent 
of~$n$.
\end{definition}
\noindent
The notion of leverage scores captures how important each row in the dataset 
is, and is used in the construction of the sampling probability.
\begin{definition}[$\ell_p$ leverage scores]
\label{def:lev}
Given $A \in \reals^{n \times d}$, suppose $U$ is an $(\alpha,\beta,p)$ 
well-conditioned basis for $\textup{range}(A)$. 
Then the $i$-th leverage score $\lambda_i$ of $A$ is defined as 
$\lambda_i = \|U_i\|_p^p$ for $i = 1,\ldots,n$.
\end{definition}

\subsection{Preconditioning}
\label{sec:precond}

Here, we briefly review the preconditioning methods that will be used in our main algorithms. A detailed summary of various preconditioning methods can be found in~\citet{YMM14_SISC,yang16pardist}. The procedure for computing a preconditioner can be summarized in the following two steps.
\begin{itemize}
\item
Given a matrix $A \in \reals^{n \times d}$ with full rank,
we first construct a sketch $SA \in \reals^{s \times d}$ for $A$ satisfying
\begin{equation}
\label{eq:low_dist}
 \sigma_S \cdot \|Ax\|_p \leq \|SAx\|_p 
    \leq \kappa_S \sigma_S \cdot \|Ax\|_p, \quad \forall x \in \reals^d,
\end{equation}
where $\kappa_S$ is the distortion factor independent of $n$.
\item
Next, we compute the QR factorization of $SA$ whose size only depends on $d$.
Return $R^{-1}$.
\end{itemize}

The following lemma guarantees that the preconditioner satisfies that $AR^{-1}$ is well-conditioned since $\kappa_S$ and $s$ depend on $d$ only, independent of $n$.
\begin{lemma}
Let $R$ be the matrix returned by the above preconditioning procedure, then we have
\begin{equation}
  \bar \kappa_p(AR^{-1}) \leq \kappa_S d^{\max\{\frac{1}{2},\frac{1}{p}\}} s^{|\frac{1}{p} - \frac{1}{2}|}.
\end{equation}
\end{lemma}

\noindent
Various ways of computing a sketching matrix $S$ satisfying \eqref{eq:low_dist} are proposed recently.
It is worth mentioning that sketching
algorithms that run in nearly input-sparsity time, i.e., in time proportional to $\bigO(\nnz(A))$
to obtain such a sketch matrix for $p=1$ and $p=2$ are available via random projections composed of sparse matrices; see \citet{CW12,MM12,WZ13,nelson2013sparse} for details.
In Tables~\ref{table:cond}~and~\ref{table:cond_l2} in Appendix~\ref{sec:sa_full} we provide a short summary of these sketching methods and the resulting running time and condition number.


\section{A connection to stochastic optimization}
\label{sxn:naive_alg}

In this section, we describe our framework for viewing deterministic $\ell_p$ regression problems from the perspective of stochastic optimization.
This framework will recover both RLA and SGD methods in a natural manner; and by combining these two approaches in a particular way we will obtain our main algorithm.

We reformulate overdetermined $\ell_p$ regression 
problems of the form \eqref{eq:lp_obj} into a stochastic optimization 
problem of the form \eqref{eq:lp_sto}~\footnote{Technically, this result is straightforward; but this 
reformulation allows us to introduce randomness---parameterized by a 
probability distribution $P$---into the deterministic problem 
\eqref{eq:lp_obj} in order to develop randomized algorithms for it.}, which reformulates a 
deterministic regression problem into a stochastic optimization problem.
Note that the result holds for general $p \in [1,\infty)$.

\begin{proposition}
\label{prop:det-lp-to-stoch-optiz}
Let $U \in \reals^{n \times d}$ be a basis of the range space of $A$ in the form $U = A F$, where $F \in \reals^{d \times d}$.
The constrained overdetermined $\ell_p$ regression 
problem~\eqref{eq:lp_obj} is equivalent to
\begin{equation}
 \label{eq:lp_sto}
  \min_{y \in \Y} \|Uy - b\|_p^p = \min_{y\in \Y}\ExpectSub{\xi \sim P}{H(y, \xi)},
\end{equation}
 where $\xi$ is a random variable over 
$\{1,\ldots,n\}$ with distribution $P=\{p_i\}_{i=1}^n$, $y$ is the decision variable in $\mathcal Y$, and $H(y,\xi) = \lvert U_\xi y - b_\xi \rvert^p/p_\xi$. The constraint set of $y$ is 
$\Y = \{ y \in \reals^{d} \vert y = F^{-1}x, x \in \Z \}$.
\end{proposition}

With Proposition~\ref{prop:det-lp-to-stoch-optiz}, as suggested in Figure~\ref{fig:flow}, one can solve the overdetermined $\ell_p$ regression problem \eqref{eq:lp_obj}
by applying either SAA or SA, i.e., ($C1$) on the stochastic optimization problem \eqref{eq:lp_sto}.
In addition to the choice of SA versus SAA, one also has to choose $U$ and $P$, i.e., ($C2$), and determine the underlying solver, i.e.,~($C3$).

Assume that if SAA is used, then for ($C3$) we solve the subproblem exactly, i.e., we compute a high-precision solution of the subproblem; this leads to a class of {\it randomized linear algebra} (RLA) algorithms for solving $\ell_p$ regression. 
Alternatively, if we assume that SA is used, then we extract the first-order information, i.e., sub-gradient of the sample, and update the weight in a gradient descent fashion; this leads to a family of {\it stochastic gradient descent} (SGD) algorithms for solving $\ell_p$ regression.

For ($C2$), we need to choose a basis $U$ that converts \eqref{eq:lp_obj} into an 
equivalent problem represented by $U$ and choose a distribution $P$ for which the algorithm samples a row 
at every iteration accordingly.
In general, different choices of $U$ and $P$ lead to different algorithms.
In the following two subsections, we will discuss their effects on SAA and SA and make connections between existing solvers and our new solution methods.
For simplicity, we assume there are no constraints, i.e., $\Z = \reals^d$ (although much of this framework generalizes to nontrivial constraints).

\subsection{Using RLA (SAA) to solve $\ell_p$ regression}
In this subsection, we briefly discuss the algorithms induced by applying SAA to \eqref{eq:lp_sto} with different choices of basis $U$ and distribution $P$ in Proposition~\ref{prop:det-lp-to-stoch-optiz}.

We first show that the choice of the basis $U$ has no effect on the resulting sampling algorithm.
Let $S \in \reals^{s\times n}$ be the equivalent sampling matrix in the sampling algorithm. That is, 
 $$S_{ij} = \begin{cases} 1/p_j ~~ \mbox{if the $j$-th row is sampled in the $i$-th iteration} \\
  0 ~~\mbox{otherwise}.
 \end{cases}$$
Then the subproblem can be cast as 
  $ \min_{y\in \Y} \|SUy - b\|_p^p, $
  which is equivalent to
  $ \min_{x\in \Z} \|SAx - b\|_p^p. $
Therefore, with a given distribution $P$, applying SAA to the stochastic optimization problem associated with any basis $U$ is equivalent to applying SAA to the original problem with matrix $A$.

Next, we discuss the effect of the choice of $P$, i.e., the sampling distribution in SAA, on the required sampling size.

\paragraph{Naive choice of $P$}
One choice of $P$ is a uniform distribution, i.e., $p_i=1/n$ for $i=1,\ldots,n$.
The resulting SAA algorithm becomes uniformly sampling $s$ rows from the original $n$ rows and solving the subproblem induced by the selected rows.
If all the rows are equally ``important'', such an algorithm can be expected to work.
However, consider the following toy example for which uniform sampling gives undesirable answers with high probability.
Suppose the first row of the matrix contains the only nonzero element in the first column of the design matrix $A$.
Since the only measurement of $x_1$ lies in the first row,
in order to recover the optimal value, namely $x_1^*$, the first row in matrix $A$ is crucial.
However, when a uniform sampling scheme is used, the sampling size required in order to sample the first row is $\Omega(n)$.
This implies that RLA with uniform sampling will fail with high probability unless the sampling size $s = \Omega(n)$.

\paragraph{Smarter choice of $P$}
In the above example, it is not hard to show that the leverage score of the first row is $1$, i.e., it is much larger than the average value of the leverage scores.
This inspires us to put more weights on ``important'' rows, i.e., rows with higher leverage scores.
An immediate solution is to define $P$ based on the leverage scores as follows:
\begin{equation*}
  p_i = \frac{\lambda_i}{\sum_{j=1}^n \lambda_j},
\end{equation*}
where $\lambda_i$ is the $i$-th leverage score of $A$ (which depends on whether one is working with $\ell_1$, $\ell_2$, or more general $\ell_p$ regression).
Applying SAA with this distribution and solving the subproblem exactly recovers the recently proposed RLA methods with algorithmic leveraging for solving overdetermined $\ell_p$ regression problems;
see \citep{Mah-mat-rev_BOOK, DDHKM09, CDMMMW12, YMM14_SISC,CW12,MM12,stat_lev} for details.
(In RLA, this is simply solving the subproblem of the original problem, but
in statistical learning theory, this has the interpretation of Empirirical Risk Minimization.)
This algorithm is formally stated in Algorithm~\ref{alg:lp_saa} in Appendix~\ref{sec:related_alg}.
We also include its approximation-of-quality results from~\citep{DDHKM09} in Appendix~\ref{sec:related_alg},
which state that the resulting approximate solution $\hat x$ produces a $(1+\epsilon)$-approximation to the objective if the sampling size $s$ is large enough.
(Note, in particular, that ``large enough'' here means that when the desired accuracy and failure probability are fixed, the required sampling size only depends on the lower dimension $d$, independent of $n$.)

\subsection{Using SGD (SA) to solve $\ell_p$ regression}
Applying SA to \eqref{eq:lp_sto} and updating the weight vector using first-order information results in a SGD algorithm.
It is not hard to show that, given $U = AF$ and $P = \{p_i\}_{i=1}^n$,
the update rule is as follows. 
Suppose the $\xi_t$-th row is sampled; then the weight vector $x_t$ is updated by
$$ x_{t+1} = x_t - \eta c_t H^{-1} A_{\xi_t}, $$
where
 $ H = \left( F F^\top \right)^{-1} \in \reals^{d\times d}$, $\eta$ is the step size, and $c_t$ is a constant that depends on $x_t$ and $\xi_t$.
 
Next, we discuss how different choices of $U$ and $P$ affect the convergence rates of the resulting SGD algorithms. For simplicity, we restrict our discussions to unconstrained $\ell_1$ regressions.

\paragraph{Naive choice of $U$ and $P$}
Consider the following choices of $U$ and $P$ that lead to undesirable convergence rates.
Let $U = A$.
If we apply the SGD with some distribution $P = \{p_i\}_{i=1}^n$, some simple arguments in the SGD convergence rate analysis lead to a relative approximation error of
\begin{equation}
\label{eq:complexity}
 \frac{f(\hat x) \!-\! f(x^\ast)}{f(\hat x)}\! =\!
  \bigO\! \left( \frac{\|x^\ast\|_2 \max_{1\leq i \leq n}\! \|A_i\|_1/p_i }{\|Ax^\ast-b\|_1} \right),
\end{equation}
where $f(x) = \|Ax-b\|_1$ and $x^\ast$ is the optimal solution.
When $\{p_i\}_{i=1}^n$ is the uniform distribution,
\eqref{eq:complexity} becomes
$\bigO \left( n \frac{\|x^\ast\|_2 \cdot M }{\|Ax^\ast-b\|_1} \right)$,
where $M = \max_{1\leq i \leq n} \|A_i\|_1$ is the maximum $\ell_1$ row 
norm of $A$.
Alternatively, if one chooses $p_i$ to be proportional to the row norms of $A$, 
i.e., $p_i = \frac{\|A_i\|_1}{\|A\|_1}$, then \eqref{eq:complexity} becomes
$\bigO \left( \frac{\|x^\ast\|_2 \cdot \|A\|_1 }{\|Ax^\ast-b\|_1} \right)$.
Consider the following scenario.
Given $A$ and $b$, we continue to append samples $(z, c)$ satisfying 
$z^\top  x^\ast = c$ and $\|z\|_2 \leq M$ to $A$ and $b$, respectively.
This process will keep $x^\ast$, $M$ and $\|Ax^\ast - b\|_1$ unchanged.
However, the value of $n$ and $\|A\|_1$ will increase.
Thus, in this case, the expected time for convergence of SGD with these naive sampling 
distributions might blow up as the size of the matrix grows.

\paragraph{Smarter choice of $U$ and $P$}
To avoid this problem, we need to precondition the linear regression problem. 
If we work with a well-conditioned basis $U$ for the range space of 
$A$ and choose the sampling probabilities proportional to the row 
norms of $U$, i.e., leverage scores of $A$, then the resulting convergence rate on the relative error 
of the objective becomes
$\bigO \left( \frac{\|y^\ast\|_2 \cdot \|U\|_1 }{\|U y^\ast-b\|_1} \right)$,
where $y^\ast$ is an optimal solution to the transformed problem.
By Definition~\ref{def:basis}, if $U$ is a well-conditioned basis,
then one obtains
$\|U\|_1 \leq \alpha$ and $\|y^\ast\|_\infty \leq \beta \|Uy^\ast\|_1$.
Since the condition number $\alpha \beta$ of a well-conditioned basis 
depends only on $d$ and since $\|U y^\ast-b\|_1 / \|U y^\ast\|_1$ is a constant, it implies that the resulting
SGD inherits a convergence rate in a relative scale that depends on $d$ and 
is independent of $n$.

The idea of using a preconditioner and a sampling distribution according to the leverage scores leads to our main algorithm.

\vspace{-2mm}
\section{Our Main Algorithm}
\label{sxn:theory}

In this section, we will state our main algorithm \pcsgd (Algorithm~\ref{alg:sa}) for solving the {\it constrained} overdetermined $\ell_1$ and $\ell_2$ regression problems.
We now summarize the main steps of our main algorithm as follows.

First, we compute a well-conditioned basis $U$ (Definition~\ref{def:basis}) for the range space of $A$ implicitly via a conditioning method; see Tables~\ref{table:cond}~and~\ref{table:cond_l2} in Appendix~\ref{sec:sa_full} for a summary of recently proposed randomized conditioning methods.
We refer this as the ``implicit'' method, i.e., it focuses on computing $R \in \reals^{d \times d}$ such that $U = A R^{-1}$.
A typical way of obtaining $R$ is via the QR decomposition of $SA$ where $SA$ is a sketch of $A$; see Appendix~\ref{sec:sa_full} for more details.


Second, we either exactly compute or quickly approximate the leverage scores (Definition~\ref{def:lev}), i.e., the row norms of $U$ as $\{\lambda_i\}_{i=1}^n$.
To compute $\{\lambda_i\}_{i=1}^n$ exactly, we have to form the matrix $U$ explicitly, which takes time $\bigO(nd^2)$.
Alternatively, we can estimate the row norms of $U$ without computing the product between $A$ and $R^{-1}$, in order to further reduce the running time; see Appendix~\ref{sec:sa_full} for more details.
We assume that $\{\lambda_i\}_{i=1}^n$ satisfy 
\begin{equation}
 \label{eq:est_lev}
 (1-\gamma) \|U_i\|_p^p \leq \lambda_i \leq (1+\gamma) \|U_i\|_p^p,
\end{equation}
where $\gamma$ is the approximation factor of estimation. When the leverage scores are exact, the approximation factor $\gamma=0$.
From that, we can define a distribution $P$ over $\{1,\ldots,n\}$ based on $\{\lambda_i\}_{i=1}^n$ as~follows:
\begin{equation}
 \label{eq:distri}
  p_i = \frac{\lambda_i}{\sum_{j=1}^n \lambda_j}.
\end{equation}

Third, in each iteration a new sample corresponding to a row of $A$ is drawn according to distribution $P$ and we apply an SGD process to solve the following equivalent problem with a specific choice of $F \in \reals^{d\times d}$:
\begin{equation}
   \label{eq:lp_sto2}
    \min_{y\in \Y} h(y) = \|A Fy - b\|_p^p = \ExpectSub{\xi \sim P}{\lvert A_\xi F y - b_\xi \rvert^p/p_\xi}.
\end{equation}
Here the matrix $F$ is called the preconditioner for the linear system being solved; see Section~\ref{sxn:F} for several choices of $F$.
Below, we show that with a suitable choice of $F$, the convergence rate of the SGD phase can be improved significantly.
Indeed, we can perform the update rule in the original domain (with solution vector $x$ instead of $y$), i.e., \eqref{eq:alg_update2}.       
Notice that
if $\Z = \reals^d$ and $F = I$, then the update rule can be simplified as
 \begin{equation}
  \label{eq:update_simple}
   x_{t+1} = x_t - \eta c_t A_{\xi_t}.
 \end{equation}
If $\Z = \reals^d$ and $F = R^{-1}$, then the update rule becomes
 \begin{equation}
   \label{eq:update_full}
   x_{t+1} = x_t - \eta c_t H^{-1} A_{\xi_t},
 \end{equation}
where $H = (R^\top  R)^{-1}$.
In the presence of constraints, \eqref{eq:alg_update2} only needs to solve an optimization problem with a quadratic objective over the same constraints whose size is independent of $n$.

Finally, the output is the averaged value over all iterates, i.e., $\bar x = \frac{1}{T} \sum_{t=1}^\top  x_t$, for $\ell_1$ regression, or the last iterate, i.e., $x_T$, for $\ell_2$ regression.

\begin{algorithm}[tb]
  \caption{\pcsgd  --- preconditioned weighted SGD for over-determined $\ell_1$ and $\ell_2$ regression}
  \label{alg:sa}
  \begin{algorithmic}[1]
    \STATE {\bfseries Input:} $A \in \reals^{n \times d}$, $b \in \reals^n$ with $\textrm{rank}(A) = d$, $x_0 \in \Z$, $\eta$ and $T$.
    
    \STATE {\bfseries Output:} An approximate solution vector to problem $\min_{x \in \Z} \: \|Ax - b\|_p^p$ for $p=1$ or $2$.
    
    
    
    \STATE Compute $R \in \reals^{d \times d}$ such that $U = A R^{-1}$ is a well-conditioned basis $U$ as described in Section~\ref{sec:precond}.
    
    \STATE Compute or estimate $\|U_i\|_p^p$ with leverage scores $\lambda_i$, for $i \in [n]$, that satisfies \eqref{eq:est_lev}. 
    
    \STATE Let $p_i = \frac{\lambda_i}{\sum_{j=1}^n \lambda_j}$, for $i \in [n]$.
    
    \STATE Construct the preconditioner $F \in \reals^{d\times d}$ based on $R$; see Section~\ref{sxn:F} for details.
    
     \FOR{$t = 0,\ldots,T$}
        \STATE Pick $\xi_t$ from $[n]$ based on distribution $\{p_i\}_{i=1}^n$.
        \STATE $$ c_t = \begin{cases} \textrm{sgn}\left( A_{\xi_t} x_t - b_{\xi_t} \right)/p_{\xi_t}
        & \mbox{if } p=1; \\
            2 \left( A_{\xi_t} x_t - b_{\xi_t} \right) /p_{\xi_t} & \mbox{if } p=2.
 \end{cases}
          $$

        \STATE Update $x$ by
         \begin{equation}\label{eq:alg_update2}
          x_{t+1} = 
         \begin{cases} 
           x_t - \eta c_t H^{-1} A_{\xi_t}   & \mbox{if } \Z = \reals^d;  \\
           \underset{x\in\Z}{\arg\min}~~\eta c_t A_{\xi_t} x + \frac{1}{2} \|x_t - x\|_H^2
             & \mbox{otherwise}.
         \end{cases}
         \end{equation}
           where 
            $ H = \left( F F^\top \right)^{-1} $.
        \ENDFOR  
          
     \STATE {\bf Return} $\bar x$ for $p=1$ or $x_T$ for $p=2$. 

  \end{algorithmic}
\end{algorithm}


\vspace{-2mm}
\subsection{Main results for $\ell_1$ and $\ell_2$ regression problems}
\label{sec:performance_alg}

The quality-of-approximation of Algorithm~\ref{alg:sa} is presented in Proposition~\ref{cor:l1_new} and Proposition~\ref{cor:l2_new} for $\ell_1$ and $\ell_2$ regression, respectively, in which we give the expected number of iterations that \pcsgd needs for convergence within small tolerance.
We show that \pcsgd inherits a convergence rate of $\bigO\left(1/\sqrt{T}\right)$ for $\ell_1$ regression and $\bigO \left( \log T/T \right)$ for $\ell_2$ regression and the constant term only depends on the lower dimension $d$ when $F=R^{-1}$.
Worth mentioning is that for $\ell_2$ regression, our bound on the solution vector is measured in prediction norm, i.e., $\|Ax\|_2$.
For completeness, we present the non-asymptotic convergence analysis of \pcsgd in Proposition~\ref{thm:expectation_bound} and Proposition~\ref{thm:l2_expectation_bound} in Appendix~\ref{sec:sa_full}.
All the proofs can be found in Appendix~\ref{sec:proofs}.
The analysis of these results is based on the convergence properties of SGD; see Appendix~\ref{sec:sgd} for technical details.

In the following results, $R$ is the matrix computed in step 3 in Algorithm~\ref{alg:sa}, $\{\lambda_i\}_{i\in[n]}$, are the leverage scores computed in step 4, $F$ is the preconditioner chosen in step 6 in Algorithm~\ref{alg:sa} and $ H = \left( F F^\top \right)^{-1} $.
Denote by $\bar \kappa_p(U)$ the condition number of the well-conditioned basis $U= A R^{-1}$ and $\gamma$ the approximation factor of the leverage scores $\lambda_i$, $i\in[n]$, that satisfies \eqref{eq:est_lev}.
For any vector $x \in \reals^d$, denote by $\|x\|_H^2 = x^\top  H x$ the ellipsoidal norm of $x$ induced by matrix $H=H^\top\succ 0$.
For any non-singular matrix $A$, denote $\kappa(A) = \|A\|_2 \|A^{-1}\|_2$ and $\hat \kappa(A) = |A|_1 |A^{-1}|_1$. 
The exact form of the step-sizes used can be found in the proofs~\footnote{The exact expression of the optimal stepsize contains unknown quantities such as $x^\ast$. In fact, this is also the case for many SGD-type algorithms. In practice, standard techniques for searching stepsizes can be used. In our experiments, we evaluate our algorithm using theoretically optimal stepsizes as well as stepsizes after grid searching.}.

\begin{proposition}
\label{cor:l1_new}
For $A \in \reals^{n\times d}$ and $b \in \reals^n$,
  define $f(x) = \|Ax-b\|_1$ and suppose $f(x^\ast) > 0$.
Then there exists a step-size $\eta$ such that
after
 \[
 T = d \bar \kappa_1^2(U) \hat \kappa^2(RF) \frac{c_1^2 c_2 c_3^2}{\epsilon^2} 
 \]
 iterations, 
  Algorithm~\ref{alg:sa} with $p=1$ returns a solution vector estimate $\bar x$ that satisfies the expected relative error bound
  \[
  \frac{ \Expect{f(\bar x)} - f(x^\ast)}{f(x^\ast)} \leq \epsilon.
  \]
   Here, the expectation is taken over all the samples $\xi_1, \ldots, \xi_T$ and $x^\ast$ is the optimal solution to the problem $\min_{x \in \Z} f(x)$. The constants in $T$ are given by
   $c_1 = \frac{1+\gamma}{1-\gamma}$, $ c_2 = \frac{\|x^\ast - x_0 \|_H^2}{\|x^\ast\|_H^2}$ and $c_3 = \|Ax^\ast\|_1 / f(x^\ast)$.
\end{proposition}

\begin{proposition}
\label{cor:l2_new}
For $A \in \reals^{n\times d}$ and $b \in \reals^n$,
  define $f(x) = \|Ax-b\|_2$ and suppose $f(x^\ast) > 0$.
Then there exists a step-size $\eta$ such that
after
  \[
  T = c_1 \bar \kappa_2^2(U) \kappa^2(RF) \cdot \log\left( \frac{2 c_2 \kappa^2(U) \kappa^2(RF)}{\epsilon} \right) \cdot \left( 1 + \frac{ \kappa^2(U) \kappa^2(RF) }{c_3 \epsilon} \right)
  \]
  iterations, Algorithm~\ref{alg:sa} with $p=2$ returns a solution vector estimate $x_T$ that satisfies the expected relative error bound
   \[
   \frac{\Expect{\| A(x_T - x^\ast) \|_2^2}}{\| A x^\ast \|_2^2} \leq \epsilon.
   \]
Furthermore, when $\Z = \reals^d$ and $F=R^{-1}$, there exists a step-size $\eta$ such that after
\[
  T = c_1 \bar \kappa_2^2(U) \cdot \log\left( \frac{c_2 \kappa^2(U)}{\epsilon} \right) \cdot \left( 1 + \frac{ 2\kappa^2(U)  }{\epsilon} \right)
  \]
  iterations, Algorithm~\ref{alg:sa} with $p=2$ returns a solution vector estimate $x_T$ that satisfies the expected relative error bound
  \[
  \frac{ \Expect{f(x_T)} - f(x^\ast)}{f(x^\ast)} \leq \epsilon.
  \]
Here, the expectation is taken over all the samples $\xi_1, \ldots, \xi_T$, and $x^\ast$ is the optimal solution to the problem $\min_{x \in \Z} f(x)$. 
The constants in $T$ are given by
   $c_1 = \frac{1+\gamma}{1-\gamma}$, $ c_2 = \frac{\|x^\ast - x_0 \|_H^2}{\|x^\ast\|_H^2}$, $c_3 = \|Ax^\ast\|_2^2 / f(x^\ast)^2$. 
\end{proposition}

The above results indicate two important properties of $\pcsgd$. First recall that the condition number~\footnote{One can show that $\bar \kappa_2$ is a scaled version of the standard condition number $\kappa$. $\bar \kappa_1$ is also related to $\kappa$ with $\bar \kappa_1 \geq \kappa / \sqrt{nd}$. This implies that in general $\bar \kappa_1$ can be large without preconditioning, e.g., the \texttt{buzz} dataset used in our experiments.} $\bar \kappa_p(U)$ of the well-conditioned basis $U$ is a polynomial of $d$ that is independent of $n$. Thus with a preconditioner $F=R^{-1}$ and an appropriate step-size in $\pcsgd$, the number of iterations $T$ required to achieve an arbitrarily low relative error only depends on the \emph{low dimension $d$} of the input matrix $A$. 
Second, $\pcsgd$ is \emph{robust} to leverage score approximations, i.e., the expected convergence rate will only be affected by a small distortion factor even when the approximation has low accuracy, such as $\gamma = 0.5$.

\noindent
{\bf Remark.}
For constrained $\ell_2$ regression, the bound is on the solution vector measured in prediction norm. By the triangular inequality, this directly implies $(\Expect{f(x_T)} - f(x^\ast)) / f(x^\ast) \leq \sqrt{c_3 \epsilon}$.

\noindent
{\bf Remark.}
Our approach can also be applied to other type of linear regression problems such as ridge regressions in which SGD can be invoked in a standard way. In this case, the ``condition number'' of SGD is lower than $\kappa$ due to the regularization term. The randomized preconditioning methods discussed in Section~\ref{sec:precond} can be used but it is an ``overkill'. More sophisticated preconditioning methods can be devised, e.g., based on ridge leverage scores~\citep{Cohen15_ridge}.

\vspace{-2mm}
\subsection{The choice of the preconditioner $F$}
\label{sxn:F}
As we can see, the preconditioner $F$ plays an important role in our algorithm.
It converts the original regression problem in \eqref{eq:lp_obj} to the stochastic optimization problem in \eqref{eq:lp_sto2}. 
From Proposition~\ref{cor:l1_new} and Proposition~\ref{cor:l2_new}, clearly, different choices of $F$ will lead to different convergence rates in the SGD phase (reflected in $\kappa(RF)$%
\footnote{It is also reflected in $\hat \kappa(RF)$; however, it depends on $\kappa(RF)$ because one can show $m_1 \kappa(RF) \leq \hat \kappa(RF) \leq m_2 \kappa(RF)$, where $m_1, m_2$ are constants derived using matrix norm equivalences.}) and additional computational costs (reflected in $H$ in \eqref{eq:alg_update2}~).

When $F = R^{-1}$, the effect of $\kappa_2(RF)$ on $T$ vanishes.
In this case, $H$ is also a good approximation to the Hessian $A^\top  A$. This is because usually $R$ is the $R$-factor in the QR decomposition of $SA$, where $SA$ is a ``sketch'' of $A$ satisfying~\eqref{eq:low_dist} that shares similar properties with $A$. Together we have $H = R^\top  R = (SA)^\top  (SA) \approx A^\top  A$. This implies \eqref{eq:update_full} is close to the Newton-type update.
However, as a tradeoff, since $H^{-1}$ is a $d\times d$ dense matrix, an additional $\bigO(d^2)$ cost per iteration is required to perform SGD update \eqref{eq:alg_update2}.

On the other hand, when $F=I$, no matrix-vector multiplication is needed in updating $x$. However, based on the discussion above, one should expect $\kappa(R) = \kappa(SA)$ to be close to $\kappa(A)$.  Then the term $\kappa(RF) = \kappa(R)$ can be large if $A$ is poorly conditioned, which might lead to undesirable performance in SGD phase. 

Besides the obvious choices of $F$ such as $R^{-1}$ and $I$, one can also choose $F$ to be a diagonal preconditioner $D$ that scales $R$ to have unit column norms.
According to \citet{diag_precond}, the condition number after preconditioning $\kappa(RD)$ is always upper bounded by the original condition number $\kappa(R)$, while the additional cost per iteration to perform SGD updates with diagonal preconditioner is only $\bigO(d)$.
In Section~\ref{sxn:experiments} we will illustrate the tradeoffs among these three choices of preconditioners empirically.

\vspace{-2mm}
\subsection{Complexities}
\label{sxn:complexity}
Here, we discuss the complexity of \pcsgd with $F=R^{-1}$.
The running time of Algorithm~\ref{alg:sa} consists of three parts.
First, for computing a matrix $R$ such that $U = AR^{-1}$ is well-conditioned, Appendix~\ref{sec:sa_full} provides a brief overview of various recently proposed preconditioning methods for computing $R$ for both $\ell_1$ and $\ell_2$ norms; see also Table~\ref{table:cond} and Table~\ref{table:cond_l2} for their running time $time(R)$ and preconditioning quality $\bar \kappa_p(U)$.
Particularly, there are several available sparse preconditioning methods that run in $\bigO(\nnz(A))$ plus lower order terms in $d$ time~\citep{CW12, MM12,nelson2013sparse,yang16pardist,WZ13}.
Second, to estimate the leverage scores, i.e., the row norms of $AR^{-1}$, 
\citet{DMMW12, CDMMMW12} proposed several algorithms for approximating the $\ell_1$ and $\ell_2$ leverage scores without forming matrix $U$. For a target constant approximation quality, e.g., $\gamma=0.5$ and $c_1 = \frac{1+\gamma}{1-\gamma}=3$, the running time of these algorithms is $\bigO(\log n \cdot \nnz(A))$. 
Third, Proposition~\ref{cor:l1_new} and Proposition~\ref{cor:l2_new} provide upper bounds for the expected algorithmic complexity of our proposed SGD algorithm when a target accuracy is fixed. 
Combining these, we have the following results.

\begin{proposition}
\label{prop:complexity}
Suppose the preconditioner in step 3 of Algorithm~\ref{alg:sa}, is chosen from Table~\ref{table:cond} or Table~\ref{table:cond_l2}, with constant probability, one of the following events holds for \pcsgd with $F=R^{-1}$.
To return a solution $\tilde x$ with relative error $\epsilon$ on the objective,
\begin{compactitem}
\item It runs in $time(R) + \bigO(\log n \cdot \nnz(A) + d^3 \bar \kappa_1(U)/\epsilon^2)$ for unconstrained $\ell_1$ regression.
\item It runs in $time(R) + \bigO(\log n \cdot \nnz(A) + time_{update} \cdot d \bar \kappa_1(U)/\epsilon^2)$ for constrained $\ell_1$ regression.
\item It runs in $time(R) + \bigO(\log n \cdot \nnz(A) + d^3 \log(1/\epsilon)/\epsilon)$ for unconstrained $\ell_2$ regression.
\item It runs in $time(R) + \bigO(\log n \cdot \nnz(A) + time_{update} \cdot d\log(1/\epsilon)/\epsilon^2)$ for constrained $\ell_2$ regression.
\end{compactitem}
In the above, $time(R)$ denotes the time for computing the matrix $R$ and $time_{update}$ denotes the time for solving the optimization problem in~\eqref{eq:alg_update2}.
\end{proposition}

\noindent
Notice that, since $time_{update}$ only depends on $d$, an immediate conclusion is that by using sparse preconditioning methods, to find an $\epsilon$-approximate solution, \pcsgd runs in $\bigO(\log n \cdot \nnz(A) + \poly(d)/\epsilon^2)$ time for $\ell_1$ regression and in $\bigO(\log n \cdot \nnz(A) + \poly(d) \log(1/\epsilon) /\epsilon)$ time for $\ell_2$ regression (in terms of solution vector in prediction norm for constrained problems or objective value for unconstrained problems).

Also, as can be seen in Proposition~\ref{prop:complexity},
for the complexity for $\ell_1$ regression, the tradeoffs in choosing preconditioners from Table~\ref{table:cond} are reflected here.
On the other hand, for $\ell_2$ regression, as all the preconditioning methods in Table~\ref{table:cond} provide similar preconditioning quality, i.e., $\kappa(U) = \bigO(1)$, $time(R)$ becomes the key factor for choosing a preconditioning method.
In Table~\ref{table:complexity_special}, we summarize the complexity of \pcsgd using various sketching methods for solving unconstrained $\ell_1$ and $\ell_2$ regression problems. The results are obtained by a direct combination of Tables~\ref{table:complexity},~\ref{table:cond}~and~\ref{table:cond_l2}.
We remark that, with decaying step-sizes, it is possible to improve the dependence on $\epsilon$ from $\log(1/\epsilon)/\epsilon$ to $1/\epsilon$~\citep{rakhlin12optimal}. 

\begin{table}[ht]
  \centering
  \small
   \begin{tabular}{c|cc}
     type  & sketch & complexity  \\
    \hline
    $\ell_1$ & Dense Cauchy~\citep{SW11} & $\bigO(nd^2 \log n + d^3 \log d + d^{\frac{11}{2}} \log^{\frac{3}{2}}d / \epsilon^2)$ \\
    $\ell_1$ & Fast Cauchy~\citep{CDMMMW12} & $\bigO(nd \log n + d^3 \log^5 d + d^{\frac{17}{2}} \log^{\frac{9}{2}}d / \epsilon^2)$ \\
    $\ell_1$ & Sparse Cauchy~\citep{MM12} & $\bigO(\nnz(A) \log n + d^7 \log^5 d + d^{\frac{19}{2}} \log^{\frac{11}{2}}d / \epsilon^2)$ \\
    $\ell_1$ & Reciprocal Exponential~\citep{WZ13} & $\bigO(\textup{nnz}(A)\log n + d^3 \log d + d^{\frac{13}{2}} \log^{\frac{5}{2}} d / \epsilon^2)$ \\
    $\ell_1$ & Lewis Weights~\citep{cohen15lewis} & $\bigO(\nnz(A) \log n + d^3 \log d + d^\frac{9}{2} \log^{\frac{1}{2}}d / \epsilon^2)$ \\
    \hline
    $\ell_2$ & Gaussian Transform & $\bigO( nd^2 + d^3\log(1/\epsilon) /\epsilon)$ \\
    $\ell_2$ & SRHT~\citep{tropp2011improved} & $\bigO( nd \log n + d^3 \log n \log d + d^3\log(1/\epsilon) /\epsilon)$  \\
    $\ell_2$ & Sparse $\ell_2$ embedding~\citep{cohen2016nearly} & $\bigO( \nnz(A) \log n + d^3 \log d + d^3\log(1/\epsilon)/\epsilon)$ \\
  $\ell_2$ & Refinement Sampling~\citep{cohen2015uniform} & $\bigO(\nnz(A)\log(n/d)\log d + d^3\log(n/d)\log d + d^3\log(1/\epsilon) /\epsilon)$
  \end{tabular}
    \caption{
    \small
    Summary of complexity of \pcsgd with different sketching methods for computing the preconditioner when solving unconstrained $\ell_1$ and $\ell_2$ regression problems.
    The target is to return a solution $\tilde x$ with relative error $\epsilon$ on the objective.
     Here, the complexity of each algorithm is calculated by setting the failure probability to be a constant.}
    \label{table:complexity_special}
 \end{table}

Finally, we remind readers that Table~\ref{table:complexity_l1}~and~\ref{table:complexity} summarize the complexities of several related algorithms for unconstrained $\ell_1$ and $\ell_2$ regression.
As we can see, \pcsgd is more suitable for finding a medium-precision, e.g., $\epsilon = 10^{-3}$, solution.
In particular, it has a dependency uniformly better than RLA methods for $\ell_1$ regression. Moreover, unlike the high-precision solvers, \pcsgd also works for constrained problems, in which case each iteration of \pcsgd only needs to solve an optimization problem with quadratic objective over the same constraints.

\vspace{-2mm}
\subsection{Complexity comparison between \pcsgd and RLA}
\label{sec:rla}

As we pointed out in Section~\ref{sxn:naive_alg}, \pcsgd and RLA methods with algorithmic leveraging (Appendix~\ref{sec:related_alg}) (RLA for short) are closely related as they can be viewed as methods using SA and SAA to solve the stochastic optimization problem~\eqref{eq:lp_sto}. Omitting the time for computing basis $U$ and sampling distribution $P$, the comparison of complexity boils down to comparing $time_{sub}(s, d)$ (for RLA) and $time_{update} \cdot T$ (for \pcsgd) where $time_{sub}(s, d)$ is the time needed to solve the same constrained regression problem with size $s$ by $d$ and $time_{update}$ denotes the time needed for to solve the optimization problem in~\eqref{eq:alg_update2}. According to the theory, for the same target accuracy, the required $s$ (sampling size) and $T$ (number of iterations) are the same asymptotically, up to logarithmic factors; see~\cite{DDHKM09,YMM14_SISC,drineas2011faster} and Section~\ref{sec:related_alg} for expression of $s$. 
When the problem is unconstrained, due to the efficiency of SGD, $time_{update} = \bigO(d^2)$ as indicated in~\eqref{eq:alg_update2}. For $\ell_2$ regression, due to the efficiency of the direct solver, $time_{sub}(s,d) = \bigO(sd^2)$. This explains why \pcsgd and RLA (low-precision solvers (sampling)) have similar complexities as shown in Table~\ref{table:complexity}. On the other hand, for unconstrained $\ell_1$ regression, a typical $\ell_1$ regression solver requires time $time_{sub}(s,d) > sd^2$. For example, if an interior point method is used~\citep{PK97}, $time_{sub}(s,d)$ is not even linear in $s$. This explains the advantage \pcsgd has over RLA as shown in Table~\ref{table:complexity_l1}. We also note that in the presence of constraints, \pcsgd may still be more efficient for solving $\ell_1$ regression because roughly speaking, $time_{sub}(s,d)/s > time_{update}$.

\vspace{-2mm}
\subsection{Connection to weighted randomized Kaczmarz algorithm}
\label{sec:connection_full}
As mentioned in Section~\ref{sec:intro}, our algorithm \pcsgd for least-squares regression is related to the weighted randomized Kaczmarz (RK) algorithm~\citep{SV09, needel-weightedsgd}.
To be more specific, 
weighted RK algorithm can be viewed as an SGD algorithm with constant step-size that exploits a sampling distribution based on row norms of $A$, i.e.,
$ p_i = \|A_i\|_2^2 / \|A\|_F^2$.
In \pcsgd, if the preconditioner $F = R^{-1}$ is used and the leverage scores are computed exactly, the resulting algorithm is equivalent to applying the weighted randomized Karczmarz algorithm on a well-conditioned basis $U = AR^{-1}$ since leverage scores are defined as the row norms of $U$. 

Since the matrix $A$ itself can be a basis for its range space, setting $U=A$ and $F=R=I$ in Proposition~\ref{cor:l2_new} indicates that weighted RK algorithm inherits a convergence rate that depends on condition number $\kappa(A)$ times the scaled condition number $\bar \kappa_2(A)$.
Notice that in \pcsgd, the preconditioning step implicitly computes a basis $U$ such that both $\kappa(U)$ and $\bar \kappa(U)$ are low. One should expect the SGD phase in \pcsgd inherits a faster convergence rate, as verified numerically in Section~\ref{sxn:experiments}.


\section{Experiments}
\label{sxn:experiments}

In this section, we provide empirical evaluations of our main algorithm \pcsgd.
We evaluate its convergence rate and overall running time on both synthetic and real datasets.
For \pcsgd, we implement it with three different choices of the preconditioner $F$.
Herein, throughout the experiments, by \pcsgd-full, \pcsgd-diag, \pcsgd-noco, we respectively mean \pcsgd with preconditioner $F=R^{-1}, D, I$; see Section~\ref{sxn:F} for details.
Note that, for \pcsgd, we use the methods from~\citet{CW12} for preconditioning. Also, we exactly compute the row norms of $AR^{-1}$ and use them as the leverage scores. In each experiment, the initial solution vector estimate is set as zero. The above algorithms are then run in the following manner.
Each epoch contains $\lceil n/10 \rceil$ iterations.
At the beginning of each epoch, 
we sample $\lceil n/10 \rceil$ indices according to the underlying distribution without replacement and update the weight using the $\lceil n/10 \rceil$ row samples from the data matrix. Finally, the plots are generated by averaging the results over $20$ independent trials.

\subsection{Empirical evaluations on synthetic datasets}
Theoretically the major advantage of $\pcsgd$ is the fast convergence rate. To evaluate its performance, we compare the convergence rate of relative error, i.e.,
$| \hat f - f^\ast|/f^\ast$, with other competing algorithms including vanilla SGD and fully weighted randomized Kaczmarz (weighted-RK) algorithm (\cite{needel-weightedsgd,SV09}) for solving least-squares problem (unconstrained $\ell_2$ regression).
For each of these methods, given a target relative error $\epsilon=0.1$ on the objective, i.e., $(\hat f - f^\ast)/f^\ast = 0.1$, we use the optimal step-size suggested in the theory.
In particular, for \pcsgd, we are showing the convergence rate of the SGD phase after preconditioning. We stop the algorithm when the relative error reaches $\epsilon$.
In this task, we use synthetic datasets for better control over the properties on input matrices $A$ and $b$. Each dataset has size $1000$ by $10$ and is generated in one of the following two ways.
\begin{itemize}
\item[\bf Synthetic 1]
The design matrix $A$ has skewed row norms and skewed leverage scores. That is, $5$ rows have leverage scores and row norms significantly larger than the rest%
\footnote{Note that, in general, there is no correlation between row norms and leverage scores unless the matrix has nearly orthonormal columns. For construction details of \texttt{Synthetic 1}, see the construction of NG matrices Section 5.3 in~\citep{yang16pardist}.}.
\item[\bf Synthetic 2]
The design matrix $A$ is of the form $A = U\Sigma V^\top $ where $U\in\reals^{1000 \times 10}$ and $V\in\reals^{10 \times 10}$ are random orthonormal matrices and $\Sigma\in\reals^{10\times 10}$ is a diagonal matrix that controls $\kappa(A)$.
\end{itemize}
In both cases, the true solution $x^\ast$ is a standard Gaussian vector and the response vector $b$ is set to be $Ax^\ast$ corrupted by some Gaussian noise with standard deviation $0.1$.

In Figure~\ref{fig:cr}, we present the results on two \texttt{Synthetic 1} datasets with condition number around $1$ and $5$.
From the plots we can clearly see that among the methods we used, \pcsgd-full and \pcsgd-diag exhibit superior speed in terms of achieving the target accuracy.
The relative ordering within \pcsgd with three different preconditioners is consistent with the theory according to our discussions in Section~\ref{sxn:F}.
Since the datasets considered here are well-conditioned, the preconditioning phases in \pcsgd-diag and \pcsgd-full have similar effects and both methods perform well. However as suggested in Corollary~\ref{cor:l2_new}, as the condition number increases (in comparison of the results in Figure~\ref{fig:cr}(a) versus Figure~\ref{fig:cr}(b)), other methods show degradations in convergence. Furthermore Figure~\ref{fig:cr}(a) shows that the weighted-RK algorithm outperforms standard SGD. This is due to the fact that $A$ in this dataset is well-conditioned but with non-uniform row norms. 
 
We further investigate the relation between the condition number of $A$ and convergence rate.
As suggested in Proposition~\ref{cor:l2_new},
for weighted SGD algorithm, the number of iterations required to solve an $\ell_2$ regression problem is proportional to $\bar \kappa_2^2(A) \kappa^2(A) = \|(A^\top  A)^{-1}\|_2^2 \|A\|_2^2 \|A\|_F^2 \leq \bar \kappa_2^4(A)$. To verify this hypothesis, we generate a sequence of $A$ matrices using \texttt{Synthetic 2} dataset with increasing $\bar \kappa^4_2(A)$ values 
such that $U$ and $V$ in the sequence are constants.%
\footnote{In \texttt{Synthetic 2}, $U$ and $V$ are fixed. $\Sigma$ is of the form $\text{diag}(\sigma_1, \ldots, \sigma_d)$ where $\sigma_i = 1 + (i-1)q$ for $i \in [d]$. We solve for $q$ such that $\sum_{i=1}^d \sigma_i^2 = \bar \kappa_2^2(U)$ for any desired value $\bar \kappa_2^2(U)$.} This construction ensures that all other properties such as leverage scores and coherence (the largest leverage score) remain unchanged. Similar to Figure~\ref{fig:cr}, we present the experimental results (number of iterations required for different methods versus $\bar \kappa_2^4(A)$) for the \texttt{Synthetic 2} dataset in Figure~\ref{fig:cond}.


As shown in Figure~\ref{fig:cond}, the required number of iterations of all the methods except for \pcsgd-full scales linearly in $\bar \kappa_2^4(A)$. This phenomenon matches the result predicted in theory. A significant advantage of \pcsgd-full over other methods is its robust convergence rate against variations in $\bar \kappa_2^4(A)$. This is mainly because its SGD phase operates on a well-conditioned basis after preconditioning and the preconditioning quality of \pcsgd-full depends only on the low-dimension of $A$; thus increasing $\bar \kappa_2^4(A)$ has little effect on changing its convergence rate. Also, while the diagonal preconditioner in \pcsgd-diag reduces the condition number, i.e., $\kappa(RD)\leq\kappa(R)$, its convergence rate still 
suffers from the increase of $\bar \kappa_2^4(A)$.

\subsection{Time-accuracy tradoeffs}
Next, we present the time-accuracy tradeoffs among these methods on 
the following two datasets described in Table~\ref{table:datasets}.
\begin{table}[H]
 \small
   \centering
   \begin{tabular}{c|ccc}
     name  &  \#rows & \# columns & $\kappa(A)$  \\
    \hline
     \texttt{Year}\tablefootnote{\url{https://archive.ics.uci.edu/ml/datasets/YearPredictionMSD}}  & $5\times 10^5$ & $90$ & $2\times 10^3$ \\
     \texttt{Buzz}\tablefootnote{\url{https://archive.ics.uci.edu/ml/datasets/Buzz+in+social+media+}} & $5\times 10^5$ & $77$ & $10^8$
    \end{tabular}
   \caption{Summary of the two real datasets we evaluate in the experiments.}
    \label{table:datasets}
 \end{table}

Here we test the performance of various methods in solving unconstrained $\ell_1$ and $\ell_2$ regression problems. Although there are no theoretical results to support the solution vector convergence on $\ell_1$ regression problems with \pcsgd, 
we still evaluate relative error in the solution vector.
To further examine the performance of \pcsgd methods, we also include AdaGrad, SVRG, and RLA methods with algorithmic leveraging (RLA for short) mentioned in Section~\ref{sxn:naive_alg} and Appendix~\ref{sec:related_alg} for comparisons.
For AdaGrad, we use diagonal scaling and mirror descent update rule. For SVRG, we compute the full gradient every $[n/2]$ iterations.
As for implementation details, in all SGD-like algorithms, step-size tuning is done by grid-searching where at each trial the algorithm is run with a candidate step-size for enough iterations. Then the step-size that yields the lowest error within $10$ seconds is used.
The time/accuracy pair at every $2000$ iterations is recorded.
For RLA, we choose $s$ from a wide range of values and record the corresponding time/accuracy pairs. The results on the two datasets are presented in Figures~\ref{fig:tradeoffs}~and~\ref{fig:tradeoffs_buzz}, respectively.

As we can see in Figures~\ref{fig:tradeoffs}~and~\ref{fig:tradeoffs_buzz}, in our algorithm \pcsgd, a faster convergence comes with the additional cost of preconditioning.
For example, the preconditioning phase of \pcsgd takes approximately $0.5$ seconds.
Nevertheless, with a faster convergence rate in a well-conditioned basis, \pcsgd-full still outperforms other methods in converging to a higher-precision solution at a given time span. As \pcsgd-diag balances convergence rate and computational cost, it outperforms \pcsgd-full at the early stage and yields comparable results to AdaGrad. As expected, due to the poor conditioning, SGD, weighted-RK, SVRG, and \pcsgd-noco suffer from slow convergence rates. As for RLA methods, they have the same first step as in \pcsgd, i.e., preconditioning and constructing the sampling distribution. For $\ell_1$ regression, to obtain a fairly high-precision solution, the sampling size has to be fairly large, which might drastically increase the computation time for solving the sampled subproblem.
This explains the advantage of \pcsgd-full over RLA methods in Figure~\ref{fig:tradeoffs}. It is worth mentioning that, although for $\ell_2$ regression our theory provides relative error bound on the solution vector measured in the prediction norm, here we also see that \pcsgd-full and \pcsgd-diag display promising performance in approximating the solution vector measured in $\ell_2$ norm.

We also notice that on \texttt{Buzz} (Figure~\ref{fig:tradeoffs_buzz}), all the methods except for \pcsgd-full and \pcsgd-diag have a hard time converging to a solution with low solution error. This is due to the fact that \texttt{Buzz} has a high condition number. The advantage of applying a preconditioner is manifested.

Finally, notice that RLA uses a high performance direct solver to solve the mid-size subsampled problem for $\ell_2$ regression. In this case
 \pcsgd methods do not show significant advantages over RLA in terms of speed. For this reason we have not included RLA results in Figure~\ref{fig:tradeoffs}(a)~and~\ref{fig:tradeoffs}(b). Nevertheless, \pcsgd methods may still be favorable over RLA in terms of speed and feasibility when the size of the dataset becomes increasingly larger, e.g., $10^7$ by~$500$.

\begin{figure}[H]
\begin{centering}
\begin{tabular}{cc}
\subfigure[$\kappa(A) \approx 1$ ]{
\includegraphics[width=0.45\textwidth]{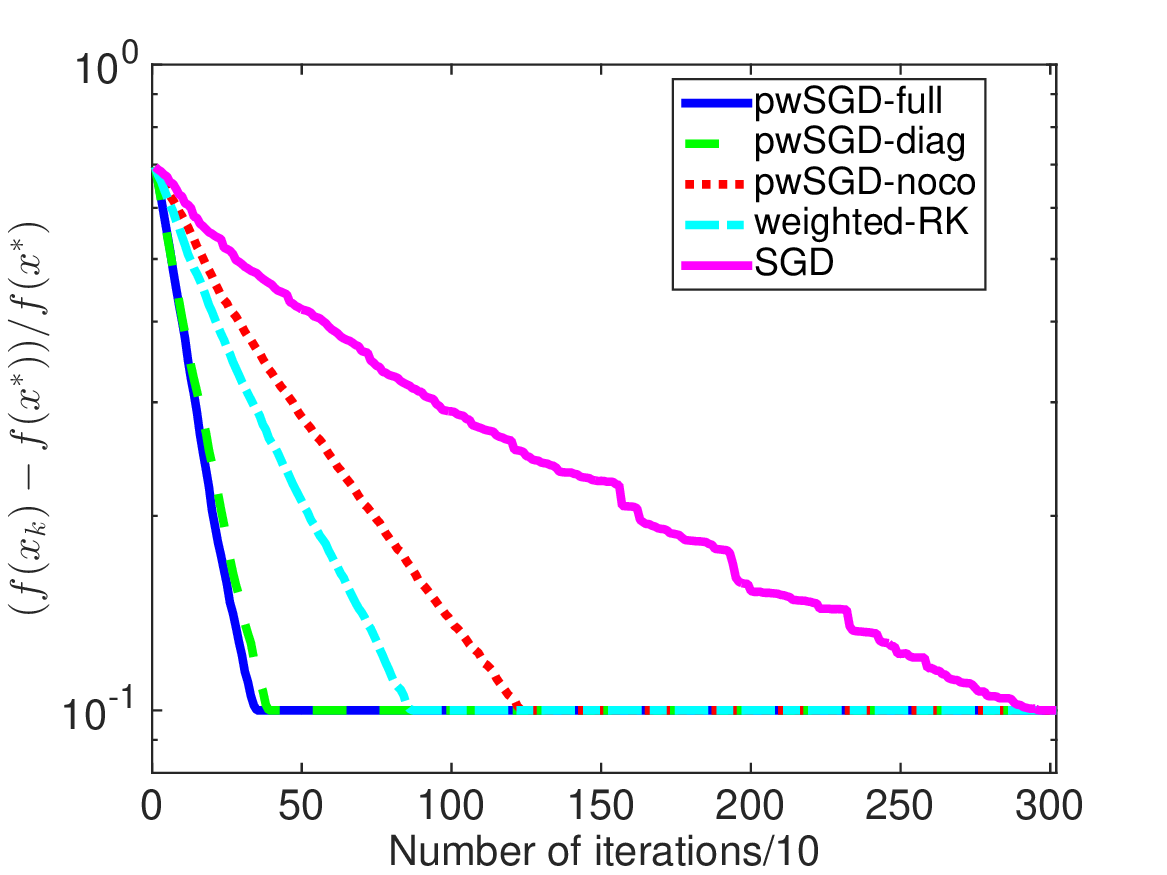}
}
&
\subfigure[$\kappa(A) \approx 5$]{
\includegraphics[width=0.45\textwidth]{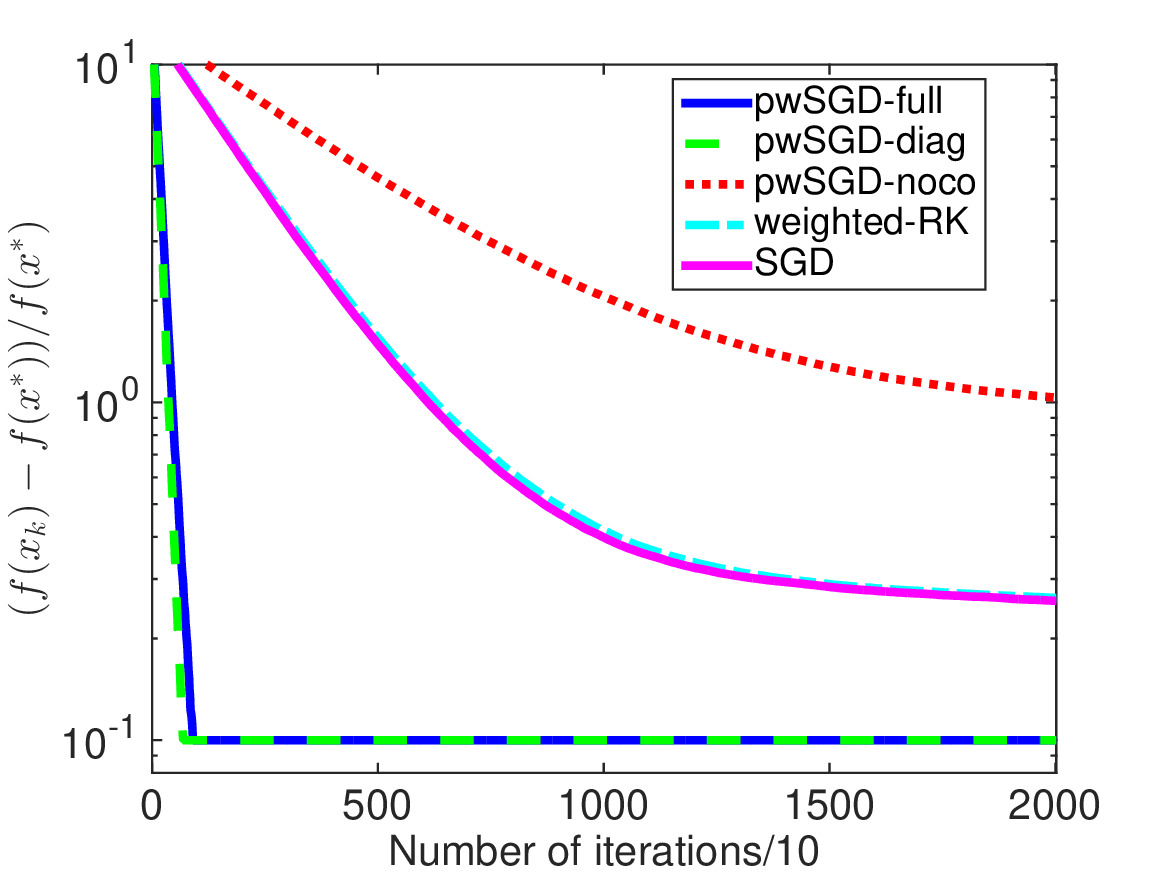}
}
\end{tabular}
\end{centering}
\caption{ Convergence rate comparison of several SGD-type algorithms including \pcsgd with three different choices of preconditioners for solving $\ell_2$ regression on \texttt{Synthetic 1} datasets with condition number around $1$ and $5$, respectively.
For each method, the optimal step-size is set according to the theory with target accuracy $|f(\hat x) - f(x^\ast)|/f(x^\ast)=0.1$.
The $y$-axis is showing the relative error on the objective,
i.e., $ |f(\hat x) - f(x^\ast)|/f(x^\ast)$.
} \label{fig:cr}
\end{figure}

\begin{figure}[H]
\centering
\includegraphics[width=0.45\textwidth]{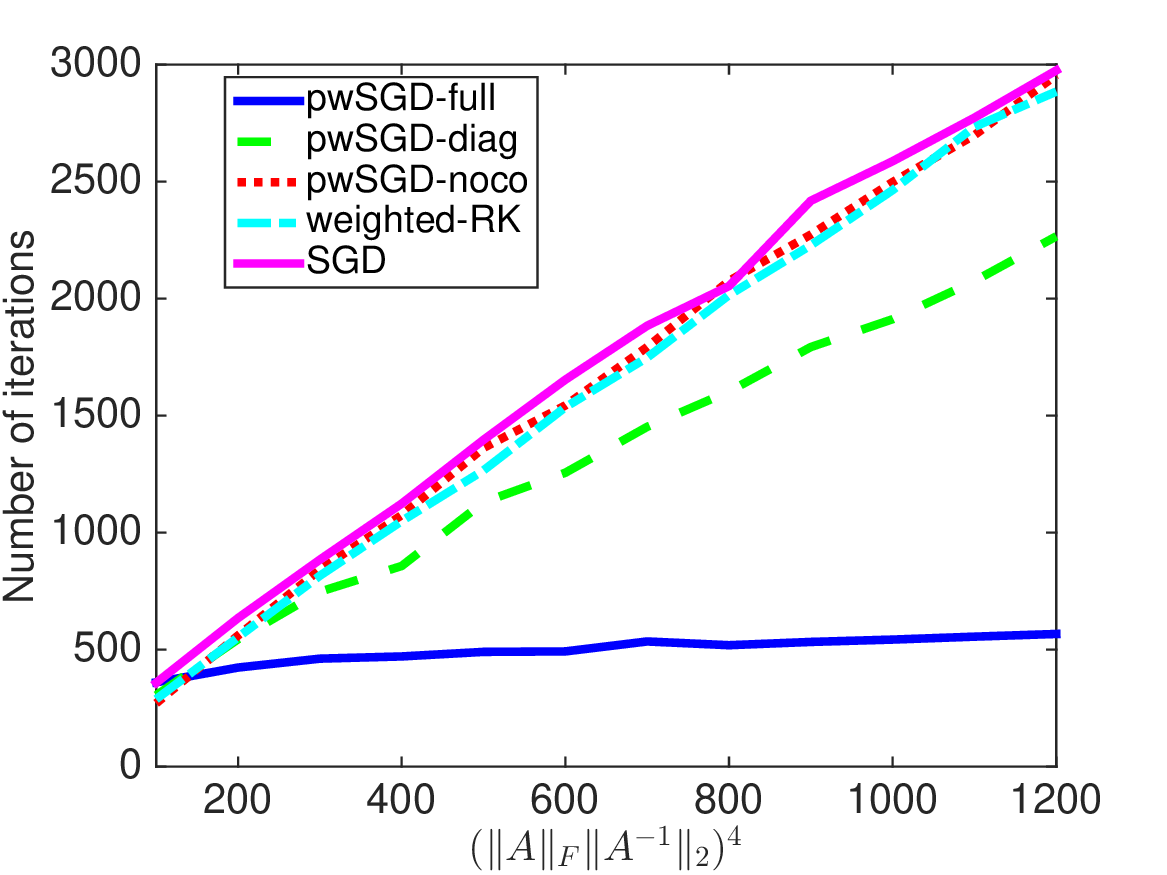}
\caption{Convergence rate comparison of several SGD-type algorithms including \pcsgd with three different choices of preconditioners for solving $\ell_2$ regression on \texttt{Synthetic 2} datasets with increasing condition number.
For each method, the optimal step-size is set according to the theory with target accuracy $|f(\hat x) - f(x^\ast)|/f(x^\ast)=0.1$.
The $y$-axis is showing the minimum number of iterations for each method to find a solution with the target accuracy.
} 
\label{fig:cond}
\end{figure}

\begin{figure}[ht]
\begin{centering}
\begin{tabular}{cc}
\subfigure[$\ell_2$~regression]{
\includegraphics[width=0.45\textwidth]{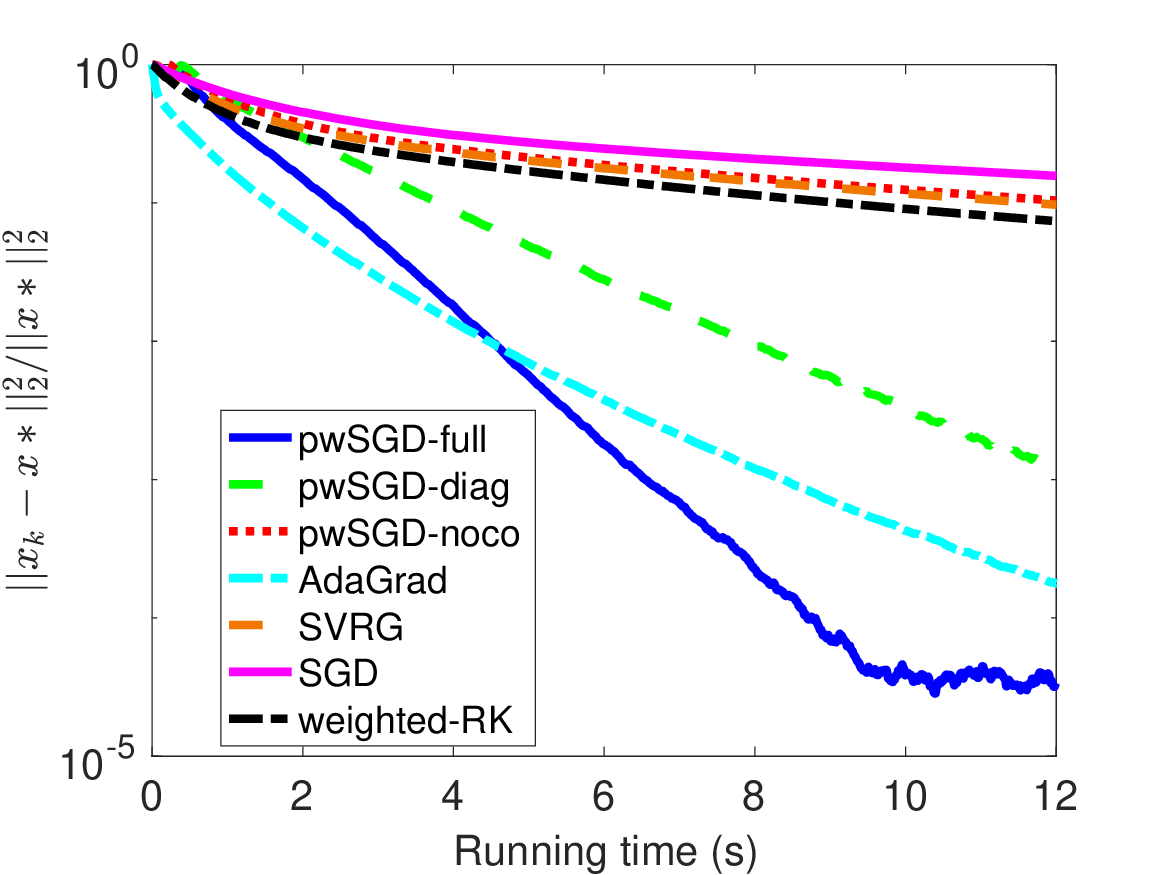}
}
&
\subfigure[$\ell_2$~regression]{
\includegraphics[width=0.45\textwidth]{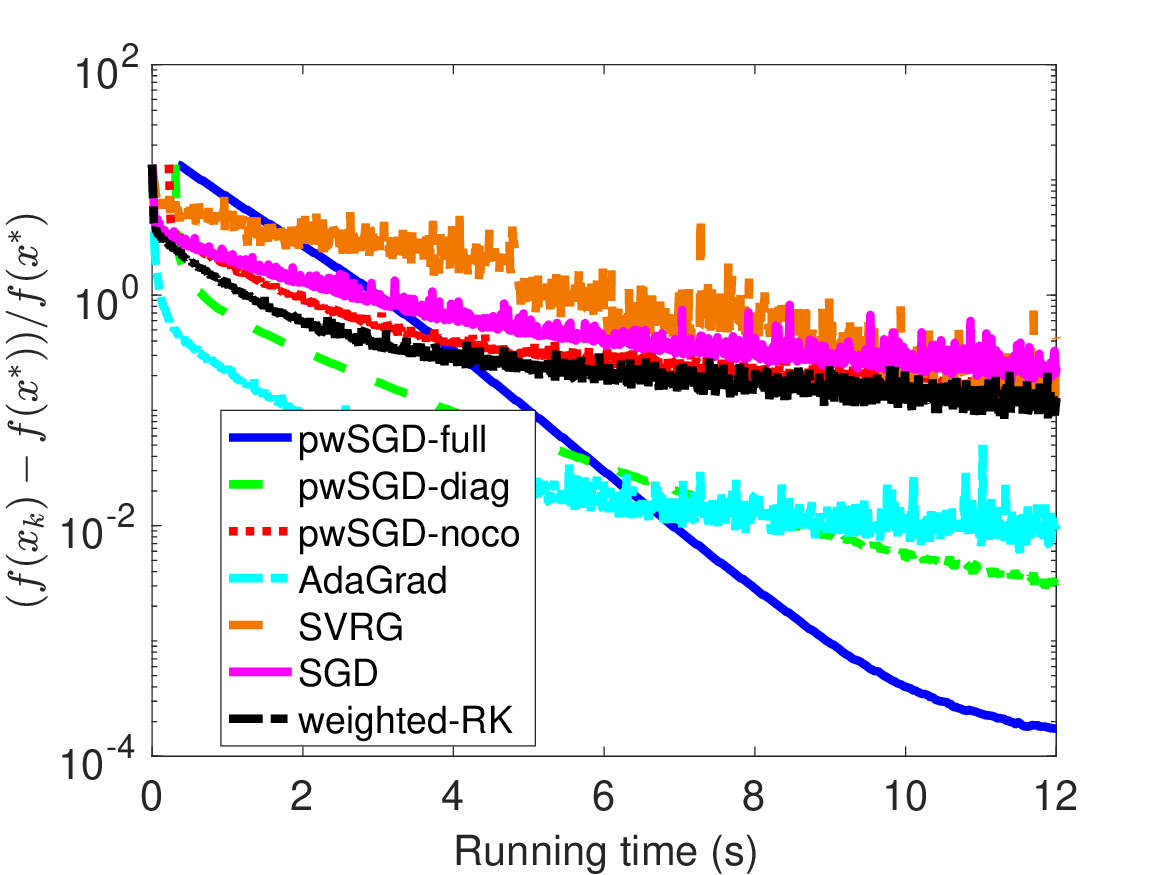}
}
\\
\subfigure[$\ell_1$~regression]{
\includegraphics[width=0.45\textwidth]{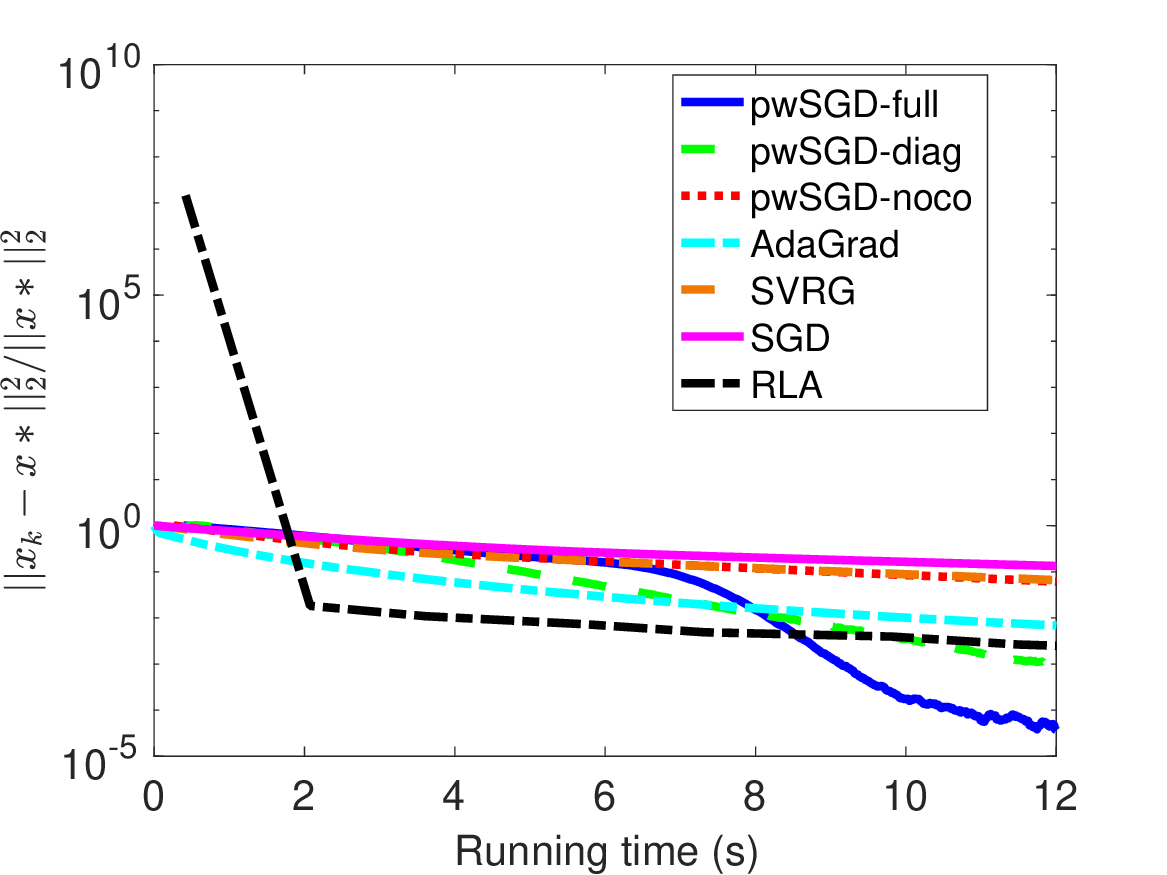}
}
&
\subfigure[$\ell_1$~regression]{
\includegraphics[width=0.45\textwidth]{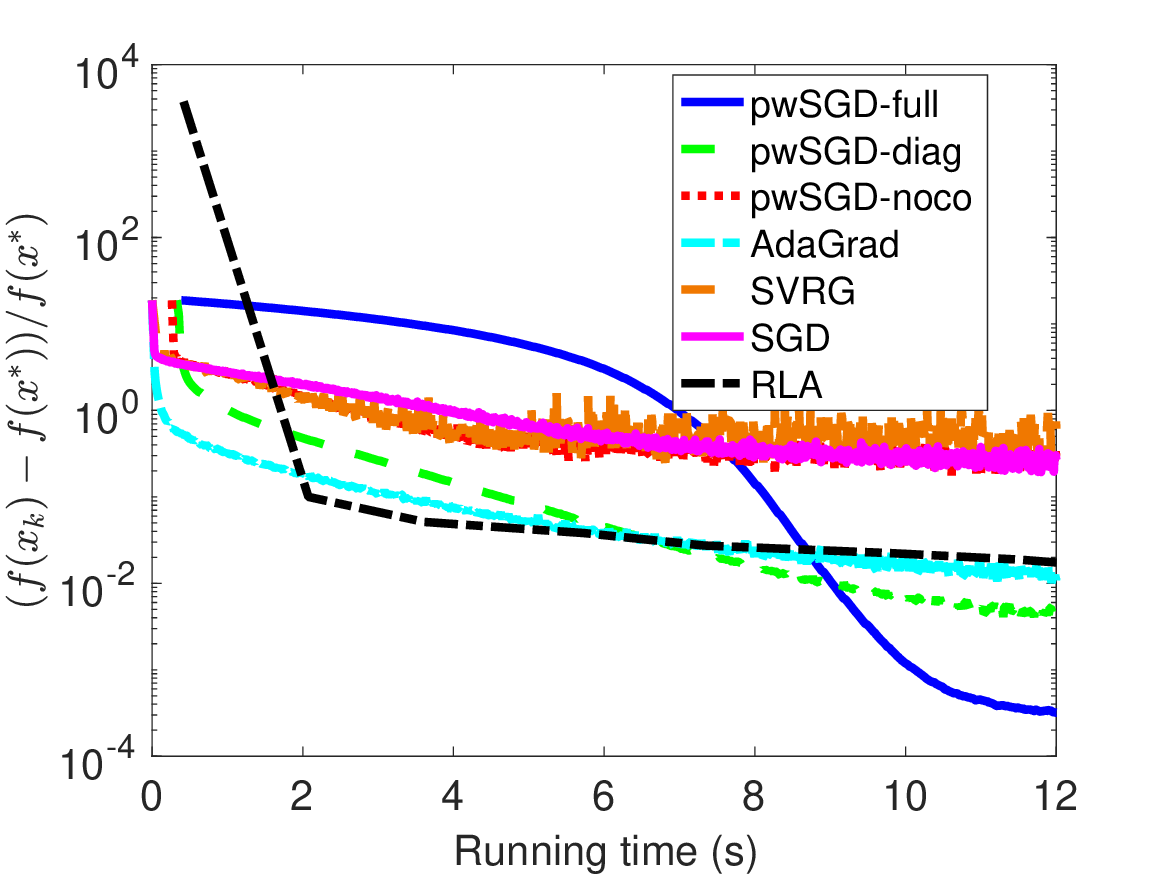}
}
\end{tabular}
\end{centering}
\caption{ Time-accuracy tradeoffs of several algorithms including \pcsgd with three different choices of preconditioners on \texttt{year} dataset. Both $\ell_1$ and $\ell_2$ regressions are tested and the relative error on both the objective value,
i.e., $ |f(\hat x) - f(x^\ast)|/f(x^\ast)$, and the solution vector, i.e., $\|\hat x - x^\ast\|_2^2 / \|x^\ast\|_2^2$, are measured.
 } 
\label{fig:tradeoffs}
\end{figure}

\begin{figure}[ht]
\begin{centering}
\begin{tabular}{cc}
\subfigure[$\ell_2$~regression]{
\includegraphics[width=0.45\textwidth]{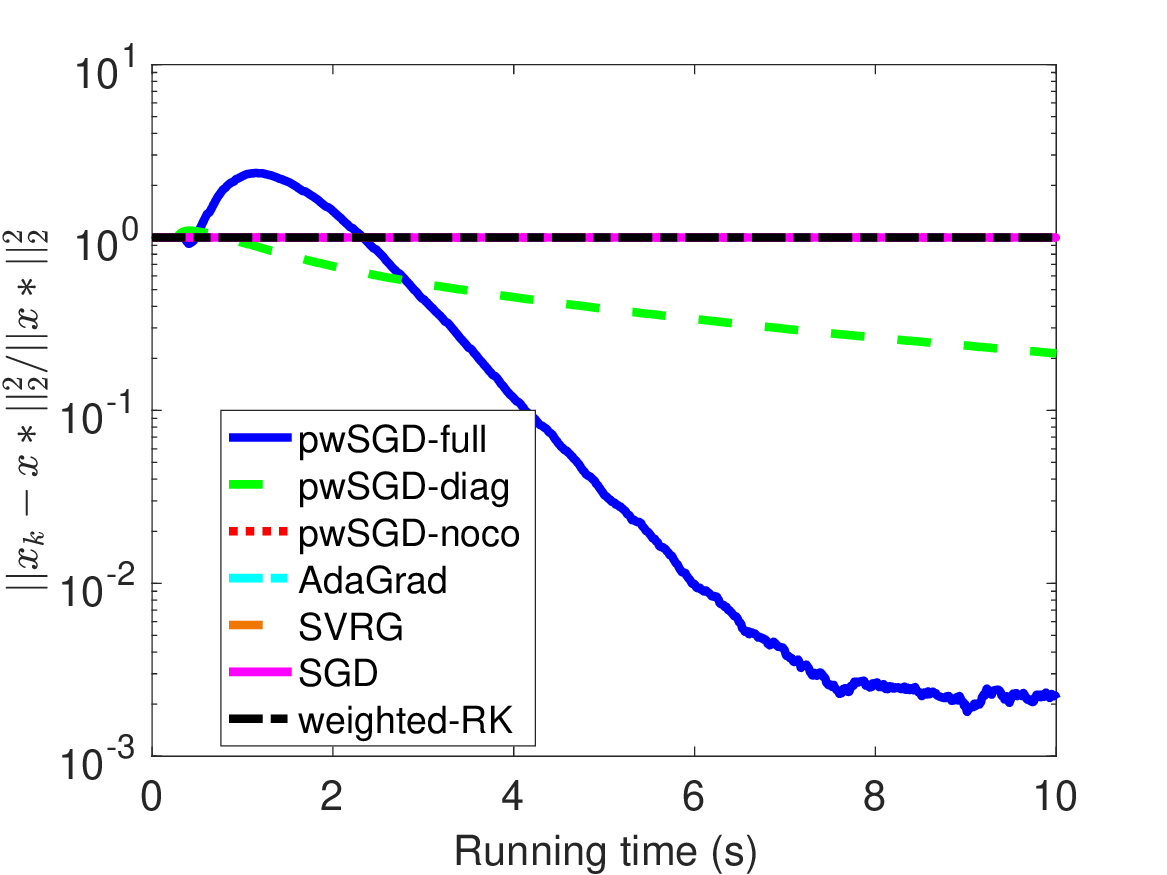}
}
&
\subfigure[$\ell_2$~regression]{
\includegraphics[width=0.45\textwidth]{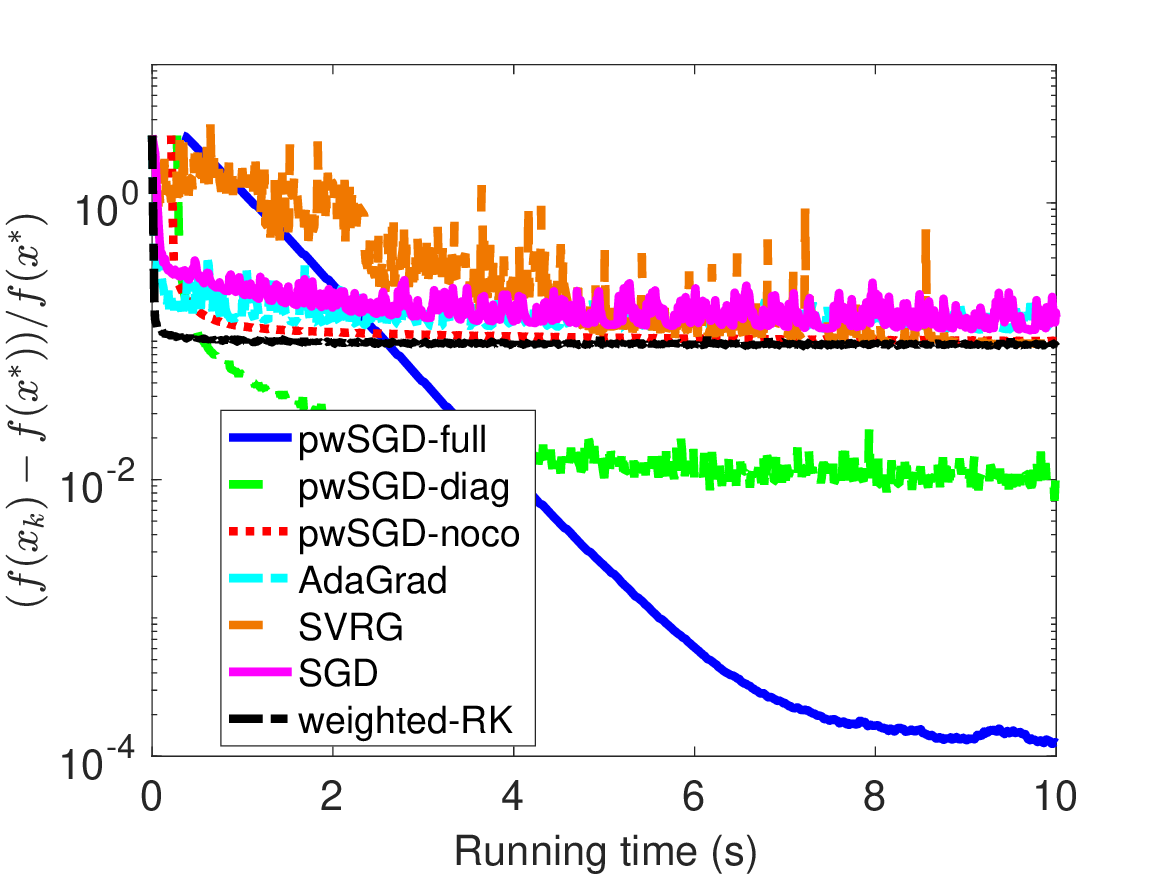}
}
\\
\subfigure[$\ell_1$~regression]{
\includegraphics[width=0.45\textwidth]{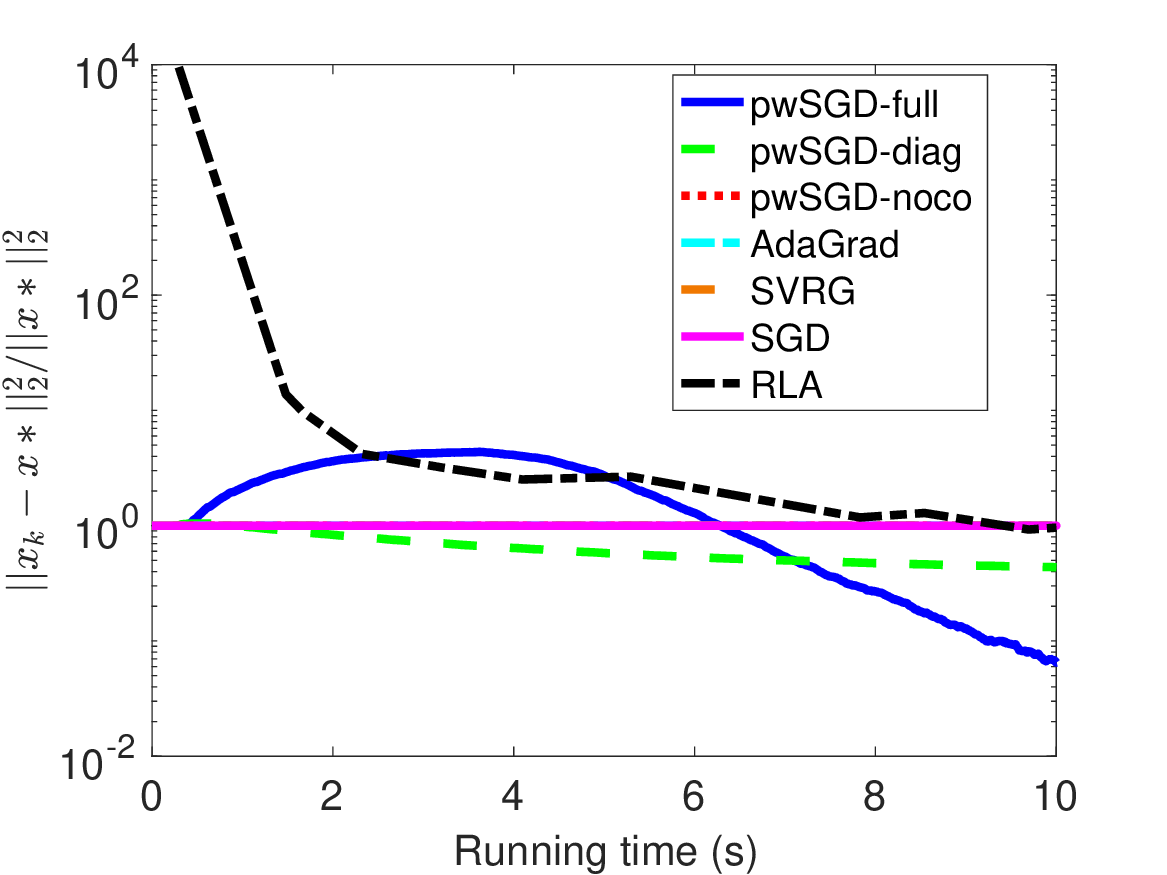}
}
&
\subfigure[$\ell_1$~regression]{
\includegraphics[width=0.45\textwidth]{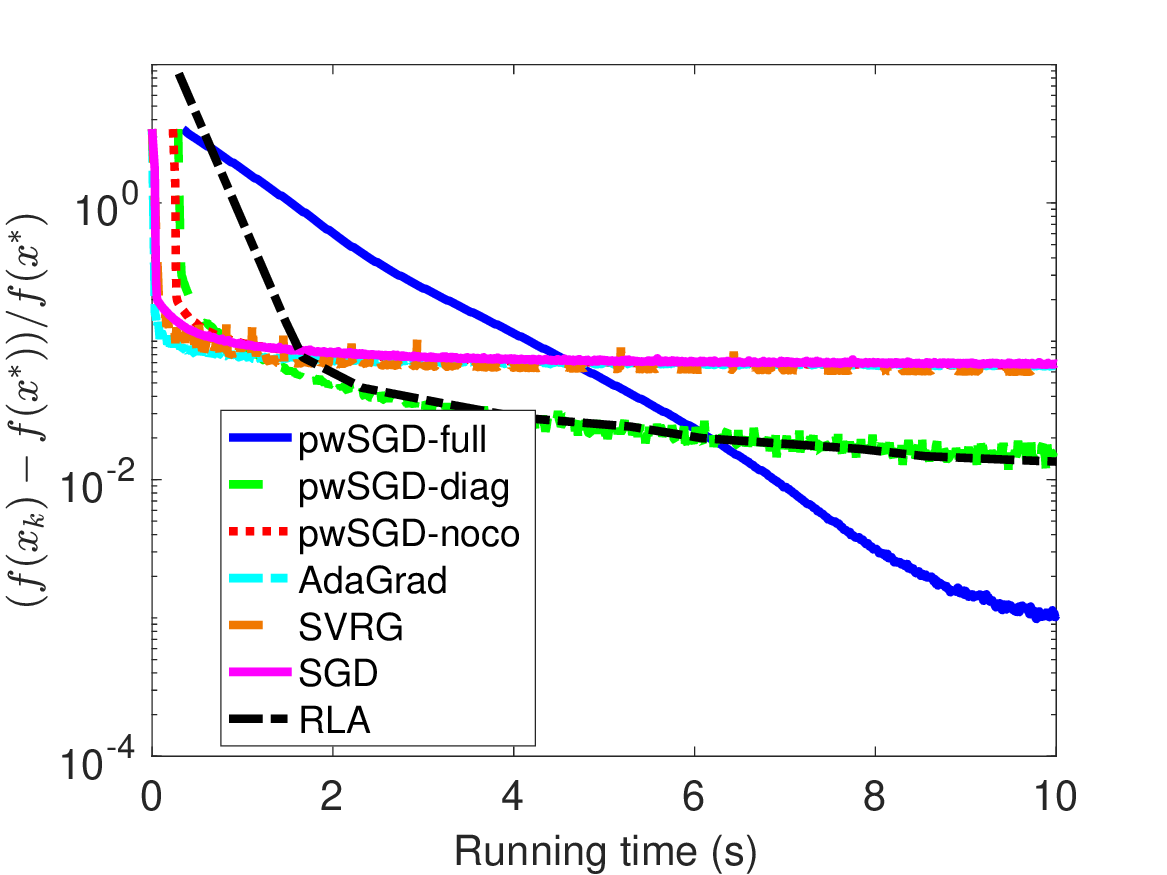}
}
\end{tabular}
\end{centering}
\caption{ Time-accuracy tradeoffs of several algorithms including \pcsgd with three different choices of preconditioners on \texttt{buzz} dataset. Both $\ell_1$ and $\ell_2$ regressions are tested and the relative error on both the objective value,
i.e., $ |f(\hat x) - f(x^\ast)|/f(x^\ast)$, and the solution vector, i.e., $\|\hat x - x^\ast\|_2^2 / \|x^\ast\|_2^2$, are measured.
 } 
\label{fig:tradeoffs_buzz}
\end{figure}

\subsection{Empirical evaluations with sparse $\ell_2$ regression}
Finally, we evaluate our algorithm on a constrained problem --- sparse $\ell_2$ regression, which is a special case of~\eqref{eq:lp_obj}. 
The problem formulation is as follows. Given a matrix $A \in \reals^{n\times d}$ and a vector $b \in \reals^n$, we want to solve the following constrained problem
  \begin{equation}
   \label{eq:lasso}
     \min_{\|x\|_1 \leq R} \|Ax - b\|_2,
  \end{equation}
where $R$ controls the size of the $\ell_1$-ball constraint. 
  
When using \pcsgd, according to~\eqref{eq:alg_update2} in Algorithm~\ref{alg:sa}, at each iteration, a sparse $\ell_2$ regression problem with size $d$ by $d$ needs to be solved. Here, to use the samples more efficiently, we use a mini-batch version of \pcsgd.
 That is, in Step 8-10 of Algorithm~\ref{alg:sa}, rather than picking only one row from $A$ to compute the noisy gradient, we select $m$ rows and average the scaled version of them.
Doing this allows us to reduce the variance of the noisy gradient.
In our experiments, we set $m = 200$. 
 
In this task, the observation model is generated in the following manner, $b = Ax^\ast + e$ where $A\in \reals^{n\times d}$ has independent standard normal entries, $x^\ast$ has $s$ nonzero entries and noise vector $e \in \reals^n$ has independent standard normal entries. 
We evaluate both the optimization error $\|\hat x - x^{LS} \|_2$ and statistical error $\|\hat x - x^\ast \|_2$ of \pcsgd-full with several choices of stepsize $\eta$ where $x^{LS}$ the optimal solution of problem~\eqref{eq:lasso}.
It is known that the least squares error of $x^{LS}$ is $\|x^{LS} - x^\ast\|_2 \approx \sqrt{s\log(ed/s)/n}$~\citep{hastie2015statistical}.
The statistical error can be bounded using the triangle inequality as shown below,
\begin{equation*}
  \|\hat x - x^\ast\|_2 \leq \|\hat x - x^{LS} \|_2 + \|x^{LS} - x^\ast \|_2.
\end{equation*}
Therefore, the statistical error $\|\hat x - x^\ast\|_2$ is dominated by the least squares error $\|x^{LS} - x^\ast \|_2$ when the optimization error $\|\hat x - x^{LS} \|_2$ is small.

In Figure~\ref{fig:lasso}, we show the results on a data instance with $n = 1e4$, $d = 400$ and $s = 30$. Here $R$ is set to be $R = \|x^\ast\|_1$ for the experimental purpose.
First, we briefly describe the effect of stepsize $\eta$. 
When a constant stepsize is used, typically, a smaller $\eta$ allows the algorithm to converge to a more accurate solution with a slower convergence rate. This is verified by Figure~\ref{fig:lasso}(a) in which the performance of \pcsgd-full with larger $\eta$'s saturates earlier at a coarser level while $\eta = 0.001$ allows the algorithm to achieve a finer solution. 
Nevertheless, as discussed above, the statical error is typically dominated by the least squares error.
For our choice of $(n,d,s)$, one can show that the least squares error $\|x^{LS} - x^\ast\|_2^2 \approx 0.01$.
Therefore, the statistical error shown in Figure~\ref{fig:lasso}(b) is around $0.01$ when the optimization error is small enough.

\begin{figure}[ht]
\begin{centering}
\begin{tabular}{cc}
\subfigure[Optimization error $\|x_k - x^{LS}\|_2^2$]{
\includegraphics[width=0.45\textwidth]{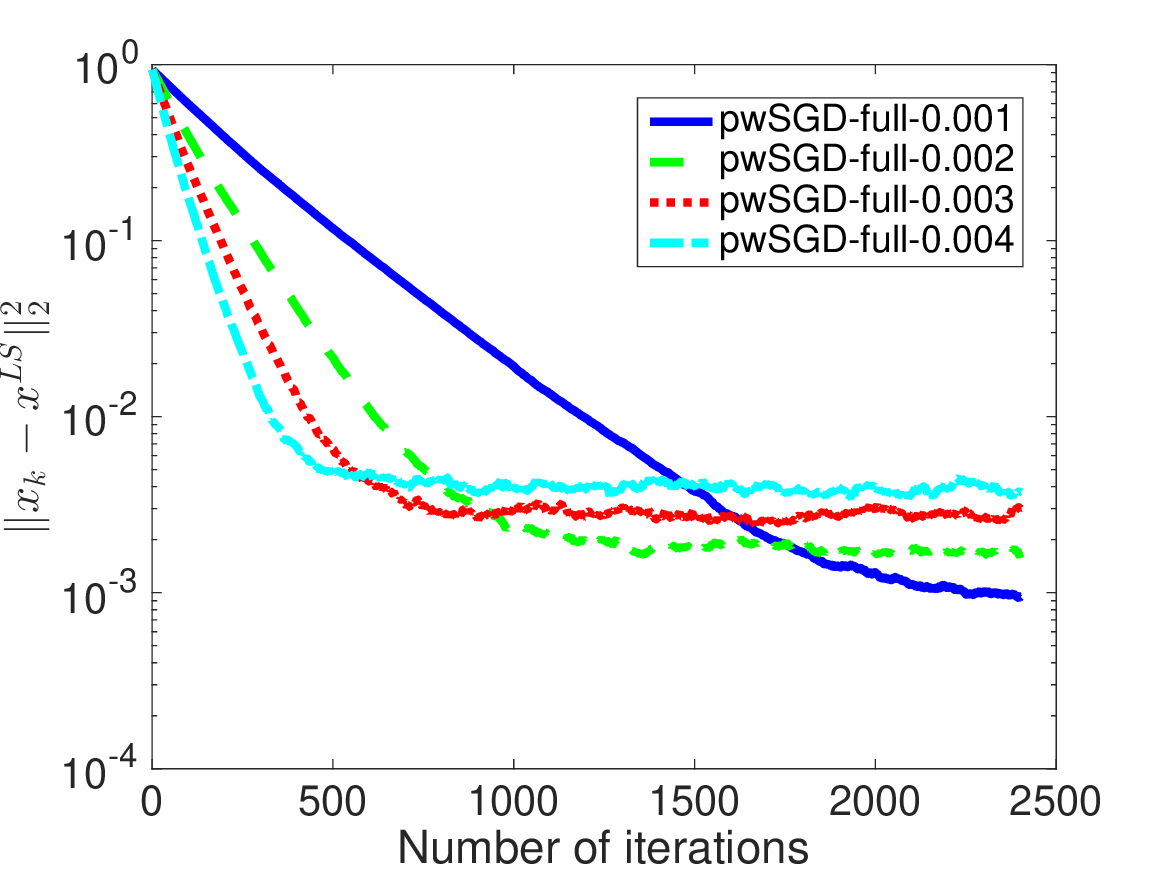}
}
&
\subfigure[Statistical error $\|x_k - x^\ast\|_2^2$]{
\includegraphics[width=0.45\textwidth]{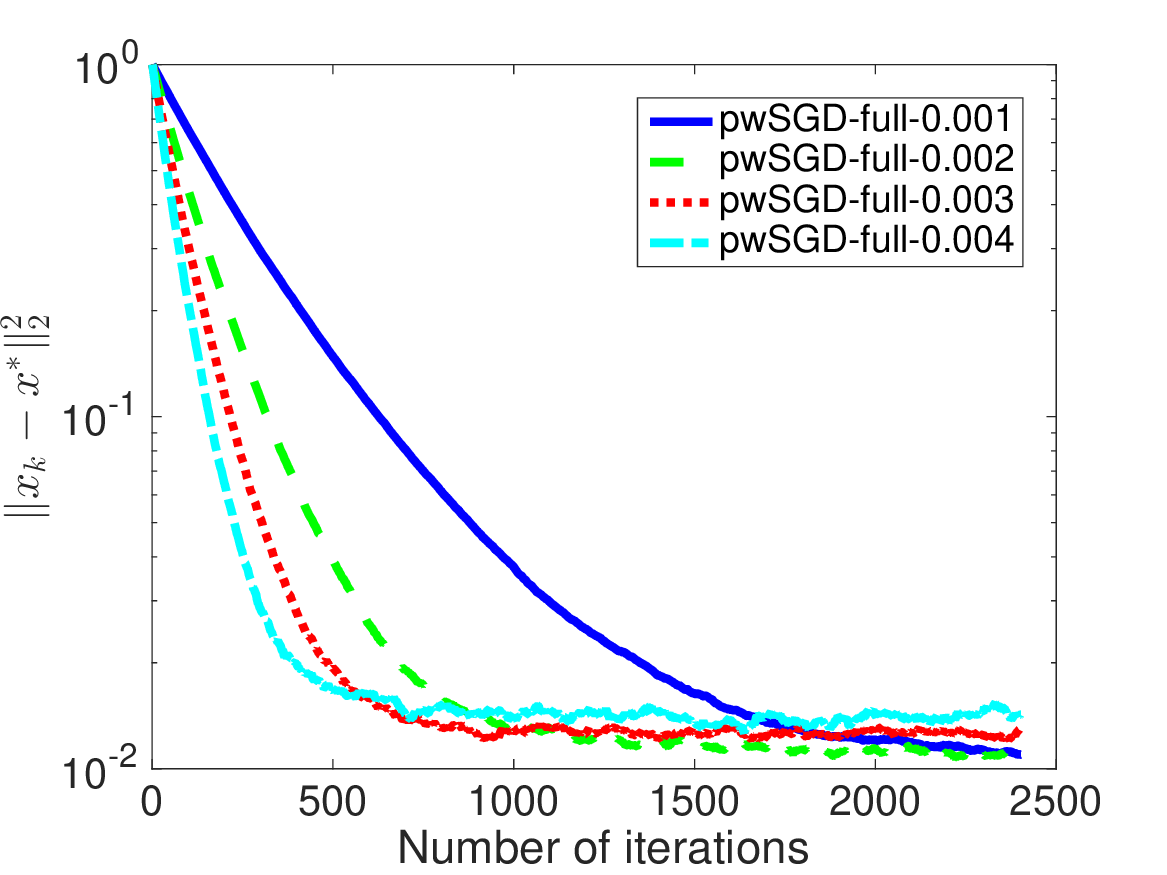}
}
\end{tabular}
\end{centering}
\caption{Performance of \pcsgd-full on a synthetic sparse $\ell_2$ regression problem with difference choices of stepsize $\eta$. Both optimization error and statistical error are shown. 
 } 
\label{fig:lasso}
\end{figure}

\section{Connection with Coreset Methods}
\label{sec:coreset}
 
After viewing RLA and SGD from the stochastic optimization perspective and 
using that to develop our main algorithm, a natural question arises: can we do 
this for other types of problems? 
To do so, we need to define ``leverage scores'' for them, since they play a 
crucial role in this stochastic framework. 
Here, we first describe the coreset framework of \citet{Feldman_coreset}.
Then we show that---on $\ell_p$ regression problems---two key notions (leverage scores from RLA and 
sensitivities from coresets) correspond. 
Finally we will show what amounts to a negative result (i.e., a lower bound) for other problems.
Note here, in this section,
we work on constrained $\ell_p$ regression $\eqref{eq:lp_obj}$ with $p \in [1,\infty)$ and we use $\bar A$ to denote the augmented linear system $\begin{pmatrix} A & b \end{pmatrix}$.

\subsection{Short summary of coreset methods}
In \citep{Feldman_coreset}, the authors propose a framework for computing a 
coreset of $\F$ to a given optimization problem of the form,
\begin{equation*}
 {\rm cost}(\F,x) = \min_{x \in \X} \sum_{f \in \F} f(x),
\end{equation*}
where $\F$ is a set of functions from a set $\X$ to $[0,\infty)$.
By Proposition~\ref{prop:det-lp-to-stoch-optiz}, it is not hard to see, the $\ell_p$ regression problem 
\eqref{eq:lp_obj} can be written as
\begin{equation*}
   \min_{x \in \C} \sum_{i=1}^n f_i(x),
\end{equation*}
where $f_i(x) = |\bar A_i x |^p$ and $\C = \{x \in \reals^{d+1} \vert x_{d+1} = -1\}$, in which case one can define a set of 
functions $\F = \{f_i \}_{i=1}^n$.  

\noindent
Central to the coreset method of \citep{Feldman_coreset} is the following 
notion of sensitivity, which is used to construct importance sampling 
probabilities, as shown in Algorithm~\ref{alg:coreset}, and the dimension of the given class of function, which is based as Definition~6.1 in \citep{Feldman_coreset}. They are defined as below.

\begin{definition}
\label{def:sensitivity}
Given a set of function $\F =\{f_i\}_{i=1}^n$, the \emph{sensitivity} $m(f)$ of each 
function is defined as
$
 m(f) =  \lfloor \sup_{x \in \X} n \cdot \frac{f(x)}{{\rm cost}(\F,x)} \rfloor + 1 ,
$
and the \emph{total sensitivity} $M(\F)$ of the set of functions is defined as
$
 M(\F) = \sum_{f \in \F} m(f).
$
\end{definition}

\begin{definition}
 \label{def:dimension}
 The dimension of $\F$ is defined as the smallest integer $d$ such that
  for any $ G \subset \F$,
 \begin{equation*}
    \lvert \{ \textup{\bf Range}(G,x,r) \mid x \in \X, r \geq 0 \} \rvert \leq \lvert G \rvert^d,  \end{equation*}
   where $ \textup{\bf Range}(G,x,r) = \{g \in G \mid g(x) \leq r\} $.
\end{definition}

The algorithm proposed in \citep{Feldman_coreset} is summarized in Algorithm~\ref{alg:coreset} below, and the corresponding result of 
quality of approximation is presented in Theorem~\ref{prop:coreset}.

\begin{algorithm}[tb]
  \caption{Compute $\epsilon$-coreset}
  \label{alg:coreset}
  \begin{algorithmic}[1]
    \STATE {\bfseries Input:} A class of functions $\F$, sampling size $s$.
    
    \STATE {\bfseries Output:} An $\epsilon$-coreset to $\F$.
    
    \STATE Initialize $\D$ as an empty set.
    
    \STATE Compute the sensitivity $m(f)$ for each function $f \in \F$.
    
    \STATE $M(\F) \gets \sum_{f\in\F} m(f)$. 
    
    \FOR{$f \in \F$}
     \STATE Compute probabilities 
         $
               p(f) = \frac{m(f)}{ M(\F) }.  
         $
    \ENDFOR  
       
    \FOR{$i = 1,\ldots,s$}
       \STATE Pick $f$ from $\F$ with probability $p(f)$.
       \STATE Add $f / (s \cdot p(f)) $ to $\D$.
    \ENDFOR
      
     \STATE
        Return $\D$.

  \end{algorithmic}
\end{algorithm}

\begin{theorem}
 \label{prop:coreset}
 Given a set of functions $\F$ from $\X$ to $[0,\infty]$,
  if $s \geq \frac{c M(\F)}{\epsilon^2} (\dim(\F') + \log\left(\frac{1}{\delta}\right))$, then with probability at least $1-\delta$,
   Algorithm~\ref{alg:coreset} returns an $\epsilon$-coreset for $\F$.
 That is,
   \begin{equation*}
    \label{eq:appr_coreset}
     (1-\epsilon) \sum_{f\in\F} f(x) \leq \sum_{f \in \D} f(x) \leq (1+\epsilon) \sum_{f\in\F} f(x),
   \end{equation*}
where 
$
 \F' = \{ f / s(f) \mid f \in \F \}
$
is a rescaled version of $\F$.
\end{theorem}


\subsection{Connections between RLA and coreset methods}

In the following, we present two results on the connection between RLA with algorithmic leveraging, i.e., with sampling based on exact or approximate leverage scores, and coreset methods.
These results originally appeared in~\citep{vx12coreset}.
We include them here and give different proofs.

The first result shows that the sensitivities are upper bounded by a constant factor times the $\ell_p$ leverage scores.
With this connection between leverage scores and sensitivities, it is not hard 
to see that applying Algorithm~\ref{alg:coreset} to $\ell_p$ regression is 
exactly the same as applying Algorithm~\ref{alg:lp_saa} (RLA sampling algorithm described in Appendix~\ref{sec:related_alg}).

\begin{proposition}
  \label{prop:sensitivity}
Given $\bar A \in \reals^{n \times (d+1)}$, let $f_i(x) = |\bar A_i x |^p$,
for $i \in [n]$.
Let $\lambda_i$ be the $i$-th leverage score of $\bar A$.
Then, the $i$-th sensitivity
$$  
  m(f_i) \leq n \beta^p \lambda_i + 1,
$$
for $i \in [n]$
and the total sensitivity
$$
    M(\F) \leq n ( (\alpha \beta)^p + 1).
$$
\end{proposition}
 
The second result is that, for the $\ell_p$ regression problem, the dimension of the class of functions $\dim(\F')$ is the same as the dimension of the subspace being considered, which is $\bigO(d)$. 
To be more specific, since all the $f \in \F'$ here are of the form $f(x) = | a^T x|^p$ for some vector $a \in \reals^d$,
we consider a broader class of functions,
namely $\A = \{ \lvert a^T x \rvert^p \mid a \in \reals^d \}$,
and compute its dimension. 
\begin{proposition}
 \label{prop:dim}
 Let $\A = \{ \lvert a^T x \rvert^p \mid a \in \reals^d \}$.
 We have
   \begin{equation*}
      \dim(\A) \leq d+1.
   \end{equation*}
\end{proposition}

With these results, in combine with Theorem~\ref{prop:coreset}, we can see that to compute a coreset $\D$, which leads to a $\left( \frac{1+\epsilon}{1-\epsilon} \right)$-approximate solution the $\ell_p$ regression using coreset method of~\citep{Feldman_coreset}, the required sampling complexity is the same (up to constants) as that of RLA sampling algorithm, as indicated by Theorem~\ref{thm:lp_saa} (assuming $\gamma=0$) in Appendix~\ref{sec:related_alg}.


\subsection{Limitation of our approach}

From the above, we see that for $\ell_p$ regression,
a small coreset whose size only depends on $d$ exists, and by solving it we can get a $(1+\epsilon)$-approximation solution.
This results in the same sampling algorithm as in RLA.
Also, the sensitivities defined in the framework can be used as a distribution when one converts a deterministic problem into a stochastic optimization problem.
We want to see whether we can extend this scheme to other problems.
Indeed, beyond $\ell_p$ regression, the coreset methods work for any kind of convex loss function~\citep{Feldman_coreset}.
However, since it depends on the total sensitivity, the size of the coreset is not necessarily small.
For RLA, this translates into requiring a very large sample size to construct a good subproblem.
For example, for hinge loss, we have the following example showing that the size of the coreset has an exponential dependency on $d$.

\begin{proposition}
\label{prop:counter_example}
 Define $f_i(x) = f(x,a_i) = (x^T a_i)^+$, where $x, a_i \in \reals^d$ for $i \in [n]$.
 There exists a set of vectors $\{a_i\}_{i=1}^d$ such that the total sensitivity of $\F = \{f_i\}_{i=1}^n$ is approximately $2^d$.
\end{proposition}

\noindent
This result indicates
that new ideas will be needed to extend similarly RLA preconditioning 
ideas to weighted SGD algorithms for other types of convex optimization problems.
This should not be surprising, since RLA methods have been developed for randomized \emph{linear} algebra problems, but it suggests several directions for follow-up work.

\section{Conclusion}
In this paper, we propose a novel RLA-SGD hybrid algorithm called \pcsgd.
We show that after a preconditioning step and constructing a non-uniform sampling distribution using RLA techniques, its SGD phase inherits fast convergence rates that only depend on the lower dimension of the input matrix.
For $\ell_1$ regression, \pcsgd displays strong advantages over RLA methods in terms of the overall complexity. For $\ell_2$ regression, it has a complexity comparable to that of several state-of-the-art solvers.
Empirically we show that \pcsgd is preferable when a medium-precision solution is desired.
Finally, we provide lower bounds on the coreset complexity for more general regression problems, which point to specific directions for future work to extend our main results.

\vspace{5mm}
\textbf{Acknowledgments.}
We would like to acknowledge
the Army Research Office,
the Defense Advanced Research Projects Agency,
and
the Department of Energy
for providing partial support for this~work.
\vspace{5mm}

\bibliographystyle{plainnat}
\bibliography{rand_alg}

\appendix
\section{Supplementary Details of Algorithm~\ref{alg:sa}}
\label{sec:sa_full}

As we discussed,
we need to compute a well-conditioned basis implicitly and estimate its row norms,
i.e., $AR^{-1}$ and $\{\lambda_i\}_{i=1}^n$ in Steps 3~and~4 in Algorithm~\ref{alg:sa}.

In Section~\ref{sec:precond} we have summarized the major steps for computing the preconditioner using sketching.
Below in Table~\ref{table:cond} we provide a short summary of preconditioning methods using various sketches along with the resulting running time and condition number.
Note that the running time here denotes the total running for computing the matrix $R$ which is the sketching time plus the time for QR factorization of the sketch.
Again, below $\bar \kappa_p(U)$ is the condition number of $U = AR^{-1}$ as defined in Definition~\ref{def:basis}
and $\kappa(U)$ is the standard condition number of $U$.

 \begin{table}[H]
 \small
   \centering
   \begin{tabular}{c|cc}
     name  &  running time & $\bar \kappa_1(U)$ \\
    \hline
      Dense Cauchy Transform~\citep{SW11}    & $\bigO(nd^2 \log d + d^3\log d)$   &  $\bigO(d^{5/2} \log^{3/2} d)$ \\
      Fast Cauchy Transform~\citep{CDMMMW12}   &  $\bigO(nd \log d + d^3\log d )$  &  $\bigO(d^{11/2} \log^{9/2} d)$ \\ 
      Sparse Cauchy Transform~\citep{MM12}  &  $\bigO(\textup{nnz}(A) + d^7 \log^5 d )$  &   $\bigO(d^{\frac{13}{2}} \log^{\frac{11}{2}} d)$ \\
      Reciprocal Exponential Transform~\citep{WZ13} & $\bigO(\textup{nnz}(A) + d^3 \log d )$  &   $\bigO(d^{\frac{7}{2}} \log^{\frac{5}{2}} d)$ \\
      Lewis Weights~\citep{cohen15lewis} &   $\bigO(\textup{nnz}(A) \log n + d^3 \log d )$  &   $\bigO(d^{\frac{3}{2}} \log^{\frac{1}{2}} d)$
     \end{tabular}
   \caption{Summary of running time and condition number, for several different $\ell_1$ conditioning methods. The failure probability of each chmethod is set to be a constant.}
    \label{table:cond}
 \end{table}
 
  \begin{table}[H]
 \small
   \centering
   \begin{tabular}{c|ccc}
     name  &  running time & $\kappa_2(U)$ & $\bar \kappa_2(U)$ \\
    \hline
      Gaussian Transform & $\bigO(nd^2)$ & $\bigO(1)$ & $ \bigO(\sqrt{d})$ \\
      SRHT~\citep{tropp2011improved} & $\bigO(nd\log n + d^3 \log n \log d)$ & $\bigO(1)$ & $\bigO(\sqrt{d})$  \\
      Sparse $\ell_2$ Embedding~\citep{CW12} & $\bigO(\textup{nnz}(A) + d^4)$ & $\bigO(1)$ & $\bigO(\sqrt{d})$ \\
    Sparse $\ell_2$ Embedding\tablefootnote{In \citep{cohen2016nearly}, the author analyzes a more general version of the original count-sketch like sparse $\ell_2$ embedding~\citep{CW12}. By setting the sparsity parameter differently, different running time complexities can be achieved.}~\citep{cohen2016nearly} & $\bigO(\textup{nnz}(A) \log d + d^3\log d)$ & $\bigO(1)$ & $\bigO(\sqrt{d})$\\
    Refinement Sampling~\citep{cohen2015uniform} & $\bigO(\nnz(A)\log(n/d)\log d + d^3\log(n/d)\log d)$ & $\bigO(1)$ & $\bigO(\sqrt{d})$
    \end{tabular}
   \caption{Summary of running time and condition number, for several different $\ell_2$ conditioning methods. Here, we assume $d \leq n \leq e^d$. The failure probability of each method is set to be a constant.}
    \label{table:cond_l2}
 \end{table}
 

Next, given the implicit representation of $U$ by $R$,
to compute the leverage scores $\|U_i\|_p^p$ exactly,
one has to compute $U$ which takes $\bigO(nd^2)$ time.
Instead of forming $U$ explicitly and ``reading off'' the row norms for computing the leverage scores, one can estimate the row norms of $U$ up to a small factor by post-multiplying a random projection matrix; see \citet{CDMMMW12,DMMW12} for the cases when $p=1,2$ respectively.
The above process can be done in $\bigO(\nnz(A) \cdot \log n)$ time.

Finally, we present two additional results regarding the non-asymptotic convergence rate of \pcsgd on $\ell_1$ and $\ell_2$ regression, respectively. Notation is similar to the one used in Proposition~\ref{cor:l1_new} and Proposition~\ref{cor:l2_new}.

\begin{proposition}
\label{thm:expectation_bound}
 For $A \in \reals^{n\times d}$ and $b \in \reals^n$,
  define $f(x) = \|Ax-b\|_1$.
  Algorithm~\ref{alg:sa} with $p=1$ returns a solution vector estimate $\bar x$ that satisfies the following expected error bound
 \begin{equation}
 \begin{split}
   \Expect{ f(\bar x)} - f(x^\ast) \leq&
    \frac{1}{2\eta T} \| x^\ast - x_1\|_H^2 + \frac{\eta}{2} \left(c_1 \alpha\|RF\|_1\right)^2.
    \end{split}
  \end{equation}
 Hereby, the expectation is taken over all the samples $\xi_1, \ldots, \xi_T$ and $x^\ast$ is an optimal solution to the problem $\min_{x \in \Z} f(x)$. The constant in the error bound is given by
   $c_1 = \frac{1+\gamma}{1-\gamma}$.
 \end{proposition}
\begin{proposition}
\label{thm:l2_expectation_bound}
 For $A \in \reals^{n\times d}$ and $b \in \reals^n$,
  define $f(x) = \|Ax-b\|_2$. Algorithm~\ref{alg:sa} with $p=2$ returns a solution vector estimate $x_T$ that satisfies the following expected error bound
 \begin{equation}
 \Expect{\| x_t - x^\ast \|_H^2} \leq\left( \frac{1 - 4\eta\left(1 - 2\eta c_1 \alpha^2\|RF\|_2^2\right)}{\beta^2 \|(RF)^{-1}\|_2^2} \right)^T  \|x_0 - x^\ast\|_H^2 + \frac{ 2 c_1 \eta \bar \kappa_2^2(U) \kappa^2(RF) h(y^\ast)}{1-2c_1\eta\alpha^2\|RF\|_2^2}.
\end{equation}
Hereby, $H=(F^{-1})^\top F^{-1}$ is the weighs of the ellipsoidal norm and the expectation is taken over all the samples $\xi_1, \ldots, \xi_T$ and $x^\ast$ is an optimal solutions to the problem $\min_{x \in \Z} f(x)$. The constant in the error bound is given by
   $c_1 = \frac{1+\gamma}{1-\gamma}$.
 \end{proposition}


\section{RLA Methods with Algorithmic Leveraging}
\label{sec:related_alg}

In this section, we present the RLA sampling algorithms with algorithmic leveraging for solving $\ell_p$ regression problems mentioned in Section~\ref{sxn:naive_alg}.
The main idea in this class of algorithms is to sample rows based on the leverage scores of $\begin{pmatrix} A & b \end{pmatrix}$ and solve the sample average approximation of the $\ell_p$ regression problem.
This method is formally stated in Algorithm~\ref{alg:lp_saa}.

The following theorem (from \citet{DDHKM09}) states that if the sampling size $s$ is large enough,
the resulting approximation solution $\hat x$ produces a $\left(\frac{1+\epsilon}{1-\epsilon}\right)$-approximation to the original solution vector.
The following theorem also shows that when the desired accuracy and confidence interval are fixed, the required sampling size only depends on the lower dimension $d$ since $\alpha$ and $\beta$ are independent of $n$.

\begin{theorem}
\label{thm:lp_saa}
Given input matrix $A \in \reals^{n \times d}$ and vector $b \in \reals^n$,
 let $\alpha,\beta$ be the condition numbers of the well-conditioned basis $U$ for the range space of $\begin{pmatrix} A & b \end{pmatrix}$ and $\gamma$ be the quality of approximation to the leverage scores satisfying \eqref{eq:est_lev}.
 Then when $\epsilon < 1/2 $ and the sampling size satisfies the following condition
\begin{equation}
 s \geq \frac{1+\gamma}{1-\gamma} \frac{(32\alpha\beta)^p}{p^2\epsilon^2} \left( (d+1)\log \left( \frac{12}{\epsilon} \right) + \log\left(\frac{2}{\delta} \right) \right), 
\end{equation}
Algorithm~\ref{alg:lp_saa} returns a solution vector $\hat x$ that  satisfies the following inequality with probability at least $1-\delta$,
\begin{equation}
  \|A\hat x-b\|_p \leq \left( \frac{1+\varepsilon}{1-\varepsilon} \right) \|Ax^* - b\|_p,
\end{equation}
where $x^* \in \Z$ is an optimal solution to the original problem $\min_{x \in \Z} \|Ax - b\|_p$.
\end{theorem}

 \begin{algorithm}[tb]
  \caption{RLA methods with algorithmic leveraging for constrained $\ell_p$ regression}
  \label{alg:lp_saa}
  \begin{algorithmic}[1]
  \STATE {\bfseries Input:} $A \in \reals^{n \times d}$, $b \in \reals^n$ with $\textrm{rank}(\bar A) = k$ where $\bar A = \begin{pmatrix} A & b \end{pmatrix}$, $\Z$ and $s > 0$.
    
    \STATE {\bfseries Output:} An approximate solution $\hat x \in \reals^d$ to problem $\textrm{minimize}_{x \in \Z} \: \|Ax - b\|_p^p$. 
    
    \STATE Compute $R \in \reals^{k \times (d+1)}$ such that $\bar A = U R$ and $U$ is an $(\alpha,\beta)$ well-conditioned basis $U$ for the range space of $\bar A$.
            
    \STATE Compute or estimate $\|U_i\|_p^p$ by $\lambda_i$ satisfying \eqref{eq:est_lev} with $\gamma$, for $i \in [n]$.
    
    \STATE Let $p_i = \frac{\lambda_i}{\sum_{j=1}^n \lambda_j}$, for $i \in [n]$.
    
     \STATE
      Let $S \in \reals^{s \times n}$ be a zero matrix.
      
    \FOR{$i = 1,\ldots,s$}
       \STATE  Pick $\xi_t$ from $[n]$ based on distribution $\{p_i\}_{i=1}^n$.       
      
       \STATE Set $S_{i, \xi_t} = \left(\frac{1}{p_{\xi_t}}\right)^{\frac{1}{p}}$.
    \ENDFOR
    
     \STATE Return $\hat x = \arg\min_{x \in \Z} \|SAx - Sb\|_p$.

  \end{algorithmic}
\end{algorithm}

\noindent
{\bf Remark.}
Compared to the RLA algorithm described in Section~\ref{sxn:naive_alg}, the algorithm described here computes the leverage scored based on a basis for the range space of the augmented linear system $\bar A = \begin{pmatrix} A & b \end{pmatrix}$ rather than $A$. One can show similar results if the basis is computed for the range space of $A$.

\noindent
{\bf Remark.}
As can be seen, the sampling size is $s = \bigO( \poly(d) \log(1/\epsilon) /\epsilon^2)$ for a target accuracy $\epsilon$.
For unconstrained $\ell_2$ regression, however, it can be shown that a sampling size $s = \bigO( \poly(d) \log(1/\epsilon) /\epsilon)$ is sufficient to compute an $\epsilon$-approximate solution; see~\cite{drineas2011faster,CW12} for details.


\section{Proofs}
\label{sec:proofs}
Here, we present the proofs of the theoretical results in the main text.


\subsection{Proof of Proposition~\ref{prop:complexity}}
Consider the following three events:
\begin{itemize}
\item $\mathcal{E}_1$: Compute a matrix $R$ such that $U = AR^{-1}$ has condition number $\bar \kappa_p$, and then compute $F = R^{-1}$ and $ H = \left( F F^\top \right)^{-1} $.
\item $\mathcal{E}_2$: Given $R^{-1}$, compute $\{\lambda_i\}_{i=1}^n$ as an estimation of row norms of $AR^{-1}$ satisfying \eqref{eq:est_lev} with $\gamma = 0.5$.
\item $\mathcal{E}_3$: For a given basis $U$ with condition number $\bar \kappa_p(U)$ and $\{\lambda_i\}_{i=1}^n$ with approximation quality $\gamma$, \pcsgd returns a solution with the desired accuracy with iterations $10T$ where $T$ is specified in Proposition~\ref{cor:l1_new} or Proposition~\ref{cor:l2_new}.
\end{itemize}

Since each preconditioning method shown in Table~\ref{table:cond} successes with constant probability, $\mathcal{E}_1$ holds with a constant probability. Also, as introduced in Appendix~\ref{sec:sa_full}, $\mathcal{E}_2$ has a constant failure probability.
Finally, by Markov inequality, we know that $\mathcal{E}_3$ holds with probability at least $0.9$.
As setting the failure probability of $\mathcal{E}_1$ and $\mathcal{E}_2$ to be arbitrarily small will not alter the results in big-O notation, we can ensure that, with constant probability, $\mathcal{E}_1 \cap \mathcal{E}_2 \cap \mathcal{E}_3$ holds. 

Conditioned on the fact that $\mathcal{E}_1 \cap \mathcal{E}_2 \cap \mathcal{E}_3$ holds, to converge to the desired solution, for $\ell_1$ regression, \pcsgd runs in $\bigO(d \bar \kappa_1(U)/\epsilon^2)$ iterations.
Since all the preconditioning methods in Table~\ref{table:cond_l2} provide $\kappa(U) = \bigO(1)$ and $\bar \kappa_2(U) = \bigO( \sqrt{d} )$, for unconstrained $\ell_2$ regression, it runs in $\bigO(d \log(1/\epsilon)/\epsilon )$ iterations. For constrained $\ell_2$ regression, since an $\epsilon$-approximate solution in terms of the solution vector measured in the prediction norm implies a $\sqrt{\epsilon}$-approximate solution on the objective,
it runs in $\bigO(d \log(1/\epsilon)/\epsilon^2 )$ iterations to return an $\epsilon$-solution in the objective value.

The overall complexity is the sum of the complexity needed in each of the above events. For $\mathcal{E}_1$, it is $time(R)$ since the time for computing $F$ and $H$ is $\bigO(d^3)$ which can absorbed into $time(R)$ and they only have to be computed once.
For $\mathcal{E}_2$, it is $\bigO(\nnz(A) \cdot \log n)$.
Finally, for $\mathcal{E}_3$, when the problem is unconstrained, $time_{update} = \bigO(d^2)$; when the problem is constrained,
$time_{update} = \poly(d)$.
Combining these, we get the complexities shown in the statement. This completes the proof.


\subsection{Proof of Proposition~\ref{thm:expectation_bound}}

The proof of this proposition is structured as follows. First we reformulate the problem using Proposition~\ref{prop:det-lp-to-stoch-optiz}. Second we show that the sequence of solution vector estimates $\{x_t\}_{t=1}^T $ in Algorithm~\ref{alg:sa} is equivalent to the solution vector estimates $\{y_t\}_{t=1}^T $ obtained by running SGD on the equivalent problem. Third, we analyze the convergence rate of $\{y_t\}_{t=1}^T $ and conclude the error bound analysis.

\paragraph{Problem reformulation}
Suppose $U$ is an $\ell_p$ well-conditioned basis for the range space of $A$
and $A = UR$ for some nonsingular matrix $R$.
Let $P$ be the distribution defined based on the estimation of the corresponding leverage scores.
That is, for $i \in [n]$,
\begin{equation}
  p_i = \frac{\lambda_i}{\sum_{j=1}^n \lambda_j},
\end{equation}
where $\lambda_i$ is an estimation of $\|U_i\|_p^p$ satisfying
\begin{equation}
 (1-\gamma) \|U_i\|_p^p \leq \lambda_i \leq (1+\gamma) \|U_i\|_p^p.
\end{equation}
This implies
\begin{equation}
 \label{eq:lambda}
  \frac{1-\gamma}{1+\gamma}\frac{\|U_i\|_p^p}{\|U\|_p^p} \leq p_i \leq 
  \frac{1+\gamma}{1-\gamma}\frac{\|U_i\|_p^p}{\|U\|_p^p}.
\end{equation}
From Proposition~\ref{prop:det-lp-to-stoch-optiz}, recall that
for any non-singular matrix $F \in \reals^{(d+1)\times(d+1)}$,
the constrained $\ell_p$ regression problem
\begin{equation}
 \label{eq:obj_rep}
  \min_{x \in \Z} f(x) := \|Ax - b\|_p^p 
\end{equation}
 can be equivalently written as the following stochastic optimization problem,
\begin{equation}
 \label{eq:lp_sto3}
    \min_{y\in \Y} h(y) = \|URFy - b\|_p^p = \ExpectSub{\xi \sim P}{\lvert U_\xi RF y - b_\xi \rvert^p/p_\xi}.
\end{equation}
Notice that by comparing to the objective function defined in \eqref{eq:lp_obj} where $f(x) = \|Ax-b\|_p$, we rewrite $f(x)$ into the form of the sum of subfunctions, i.e., $f(x) = \|Ax-b\|_p^p$, so that SGD can be applied.

\paragraph{Equivalence of sequences}

By using the following linear transformation, one notices that the sequence $\{x_t\}_{t=1}^T $ obtained by \eqref{eq:alg_update2} in Algorithm~\ref{alg:sa} has a one-to-one correspondence to the sequence $\{y_t\}_{t=1}^T $ obtained by running SGD on problem \eqref{eq:lp_sto3}:
\begin{eqnarray}
 \label{eq:var_link}
    F y_t &=&  x_t ,\nonumber \\
   F \bar y &=& \bar x, \nonumber \\
   F y^\ast &=& x^\ast.
\end{eqnarray}
Thus with condition \eqref{eq:var_link}, immediately the objective function value has the following equivalence as well:
\begin{eqnarray}
 \label{eq:obj_link}
  h(y_t) &=& f(x_t), \nonumber \\
  h(\bar y) &=& f(\bar x), \nonumber \\
  h(y^\ast) &=& f(x^\ast),
\end{eqnarray}
where $\bar x = \frac{1}{T} \sum_{i=1}^T  x_t$,  
$\bar y = \frac{1}{T} \sum_{i=1}^T  y_t$ and $x^\ast$ and $y^\ast$ are the optimal point to optimization problem \eqref{eq:obj_rep} and \eqref{eq:lp_sto3} respectively.

Now we prove \eqref{eq:var_link} by induction. By defining $F y_0 = x_0$, one immediately shows that the equivalence condition holds at the base case ($t=0$). Now by induction hypothesis, assume \eqref{eq:var_link} holds for case $t=k$. Now for $t=k+1$, we show that $x_{k+1}$ returned by Algorithm~\ref{alg:sa} and $y_{k+1}$ returned by the update rule of SGD satisfy \eqref{eq:var_link}.

For simplicity, assume that at $k$-th iteration, the $i$-th row is picked.
For subfunction $h_k(y) = \lvert U_i RF y \rvert^p - b_i/p_i$, its (sub)gradient is
\begin{equation}
g_k(y) = p \cdot \mbox{sgn}(U_i RFy - b_i) \cdot (U_i RFy - b_i )^{p-1} \cdot U_i RF / p_i,
\end{equation}
 for which the SGD update rule becomes
\begin{equation}
\label{eq:update_y_l1}
 y_{k+1} = \arg\min_{y \in \Y} \eta \langle y-y_k, c_k U_i RF \rangle + \frac{1}{2}\|y_k - y\|_2^2,
\end{equation}
where $c_k = p \cdot \mbox{sgn}(U_i RFy- b_i) \cdot (U_i RFy- b_i)^{p-1} / p_i$ is the corresponding (sub)gradient.
Recall the linear transformation $F y_k = x_k$, feasible set $\Y = \{ y \in \reals^{k} \vert y = F^{-1}x, x \in \Z\}$ and input matrix $ A_i = U_i R$,
the update rule \eqref{eq:update_y_l1} becomes
\begin{equation}
  \label{eq:update_x_l1}
  x_{k+1} = \arg\min_{x \in \Z} \eta c_k A_i x + \frac{1}{2}\| F^{-1} (x_k - x)\|_2^2.
 \end{equation}
The equation above is exactly the update performed in \eqref{eq:alg_update2}.
In particular, when $\Z = \reals^d$, i.e., in the unconstrained case,
\eqref{eq:update_x_l1} has a closed-form solution as shown in \eqref{eq:alg_update2}.
From the above analysis on the equivalence between \eqref{eq:update_y_l1} and \eqref{eq:update_x_l1}, one notices $x_{k+1}$ and $y_{k+1}$ satisfy the relationship defined in \eqref{eq:var_link}, i.e., the induction hypothesis holds at $t=k+1$.

Therefore by induction, we just showed that condition \eqref{eq:var_link}, and therefore condition \eqref{eq:obj_link}, hold for any $t$.

\paragraph{Convergence rate}
Based on the equivalence condition in \eqref{eq:obj_link}, it is sufficient to analyze the performance of sequence $\{y_t\}_{t=1}^T $. When $p=1$, the objective function is non-differentiable. Thus by substituting the subgradient of an $\ell_1$ objective function to the update in \eqref{eq:update_y_l1}, one notices that the SA method simply reduces to stochastic subgradient descent. 
We now analyze the convergence rate of running stochastic subgradient descent on problem \eqref{eq:lp_sto3} with $p = 1$.

Suppose the $i$-th row is picked at the $t$-th iteration. Recall that the (sub)gradient of the sample objective $|U_i RFy - b_i|/p_i$ in \eqref{eq:update_y_l1} is expressed as
  \begin{equation}
   \label{eq:l1_g}
     g_t(y) = \mbox{sgn}(U_i RFy - b_i) \cdot U_i RF  / p_i.
  \end{equation}
 Hence, by inequality \eqref{eq:lambda}, the norm of $ g_t(y) $ is upper-bounded as follows:
   \begin{eqnarray}
    \| g_t(y) \|_1 &=& \| U_iRF \cdot \mbox{sgn}(U_iRF y- b_i)\|_1 / p_i 
     \nonumber \\
    &\leq& |RF|_1 \|U_i\|_1 \frac{1+\gamma}{1-\gamma} \cdot \frac{  |U|_1}{ \|U_i\|_1 } \leq \alpha |RF|_1 \frac{1+\gamma}{1-\gamma}.
   \end{eqnarray}
In above, we use the property of the well-conditioned basis $U$.
Furthermore by Proposition~\ref{prop:expe} and the equivalence condition in \eqref{eq:obj_link}, for $ H = \left( F F^\top \right)^{-1} $ we have
  \begin{eqnarray}
   \Expect{ f(\bar x)} - f(x^\ast) &=&
   \Expect{ h(\bar y)} - h(y^\ast) \\
   &\leq& \frac{1}{2\eta (T+1)} \|y^\ast - y_0\|_2^2 + \frac{\eta}{2} \left(\alpha|RF|_1\frac{1+\gamma}{1-\gamma}\right)^2 \nonumber \\
    &=& \frac{1}{2\eta (T+1)} \|x^\ast - x_0\|_H^2 + \frac{\eta}{2} \left(\alpha|RF|_1\frac{1+\gamma}{1-\gamma}\right)^2,
  \end{eqnarray}
  which completes the proof.


\subsection{Proof of Proposition~\ref{cor:l1_new}}

By Proposition~\ref{prop:expe},
when the step-size equals to
\[
\eta = \frac{\|y^\ast - y_0\|_2}{\alpha|RF|_1\sqrt{T+1}}\frac{1-\gamma}{1+\gamma},
\]
the expected error bound is given by
\begin{eqnarray}
   \label{eq:opt1}
   \Expect{ h(\bar y)} - h(y^\ast) &\leq& 
     \alpha |RF|_1 \frac{\|y^\ast - y_0\|_2}{\sqrt{T+1}} \frac{1+\gamma}{1-\gamma}.
  \end{eqnarray}
By simple algebraic manipulations, we have that
  \begin{eqnarray}
    \frac{1}{\sqrt{d}} \|y^\ast\|_2 
          &\leq& \|y^\ast\|_\infty = \| (RF)^{-1} RF y^\ast\|_\infty 
          \leq | (RF)^{-1} |_1 \|RFy^\ast\|_\infty  \nonumber \\
          &\leq& \beta  | (RF)^{-1} |_1 \|URF y^\ast \|_1
        = c_3 \beta |(RF)^{-1}|_1 h(y^\ast),
  \end{eqnarray}
where $c_3 = \|URFy^\ast\|_1 / h(y^\ast) $.
In above, we use the property of the well-conditioned basis $U$.

Furthermore from inequality \eqref{eq:opt1} and the equivalence condition in \eqref{eq:obj_link}, the expected relative error bound can be upper-bounded by
   \begin{eqnarray}
   \label{eq:tmp6}
   \frac{\Expect{ f(\bar x)} - f(x^\ast)}{f(x^\ast)} &=&
   \frac{\Expect{ h(\bar y)} - h(y^\ast)}{h(y^\ast)} \nonumber \\ 
   &\leq& \frac{c_3 \sqrt{d} \beta |(RF)^{-1}|_1}{\|y^\ast\|_2} \left( \alpha |RF|_1 \frac{\|y^\ast - y_0\|_2}{\sqrt{T+1}} \frac{1+\gamma}{1-\gamma}\right) \nonumber \\
    &\leq& |RF|_1 |(RF)^{-1}|_1 \frac{ \|y^\ast - y_0\|_2 }{ \|y^\ast\|_2} 
    \left( \frac{ c_3 \sqrt{d} \alpha\beta }{\sqrt{T+1}} \frac{1+\gamma}{1-\gamma}  \right).
  \end{eqnarray}
Since the right hand side of the above inequality is a function of stopping time $T>0$, for any arbitrarily given error bound threshold $\epsilon>0$, by setting the right hand side to be $\epsilon$, one obtains the following stopping condition:
\begin{equation}
 \frac{\sqrt{d} \alpha\beta }{\sqrt{T+1}} =\frac{\epsilon}{c_1 c_3 \sqrt{ c_2 }
 |RF|_1 |(RF)^{-1}|_1},
\end{equation}
where the above constants are given by 
\[
c_1 = \frac{1+\gamma}{1-\gamma},\,\, c_2 = \frac{\|x_0-x^\ast\|_H^2}{\|x^\ast\|_H^2} = \frac{\|y_0 - y^\ast \|_2^2}{\|y^\ast\|_2^2}.
\]
Rearranging the above terms we know that after
\begin{equation}
  T \geq \frac{ d \alpha^2 \beta^2 c_1^2 c_2 c_3^2 }{\epsilon^2} |RF|_1^2 |(RF)^{-1}|_1^2
\end{equation}
iterations, the relative expected error is upper-bounded by $\epsilon>0$, i.e.,
\begin{equation}
   \frac{\Expect{ f(\bar x)} - f(x^\ast)}{f(x^\ast)} \leq \epsilon.
\end{equation}
This completes the proof.


\subsection{Proof of Proposition~\ref{thm:l2_expectation_bound}}
Similar to the proof of Proposition~\ref{thm:expectation_bound}, the proof of this proposition is split into three parts: \textbf{Problem reformulation}, \textbf{Equivalence of sequences} and \textbf{Convergence rates}.
From the proof of Proposition~\ref{thm:expectation_bound}, one notices that the proofs in {\bf Problem reformulation} and {\bf Equivalence of sequences} hold for general $p$, and thus the proofs hold for the case when $p=2$ as well. Now we proceed to the proof of the convergence rate. Again by the equivalence condition, we can show the convergence rate of solution vector estimate $\{ x_t \}$ by showing the convergence rate achieved by the sequence $\{ y_t \}$, i.e., the convergence rate of SGD of problem~\eqref{eq:lp_sto3} for $p=2$.

Throughout the rest of the proof, we denote
\begin{equation}
f(x) = \|Ax-b\|_2^2, \quad h(y) = \|AFy - b\|_2^2.
\end{equation}
Denote by $ H = \left( F F^\top \right)^{-1} $ the weighs of the ellipsoidal norm. Also recall that when the leverage scores satisfy the error condition in \eqref{eq:est_lev}, we have the following condition
\begin{equation}
\frac{1-\gamma}{1+\gamma}\frac{\|U_i\|_2^2}{\|U\|_2^2} \leq p_i \leq 
  \frac{1+\gamma}{1-\gamma}\frac{\|U_i\|_2^2}{\|U\|_2^2}.
\end{equation}
Also, we assume that $U$ is $(\alpha,\beta)$-conditioned with $\bar \kappa_2(U) = \alpha\beta$. Based on Definition~\ref{def:basis}, we have
\begin{eqnarray}
\label{eq:alpha}
   \alpha^2 &=& \|U\|_F^2, \\
\label{eq:beta}
   \beta^2 &=& \| (U^\top U)^{-1} \|_2,
\end{eqnarray}
and thus
 \begin{equation}
  \label{eq:cond_u}
    \bar \kappa_2^2(U) = \| (U^\top U)^{-1} \|_2 \cdot \|U\|_F^2=\alpha^2\beta^2.
 \end{equation}

Before deriving the convergence rate, we compute a few constants.
\begin{equation}
\label{eq:mu}
\mu = 2\sigma_{\min}^2(AF) = \frac{2}{ \left\|\left( (URF)^\top URF \right)^{-1} \right\|_2^2} \geq \frac{2}{ \left\|\left(U^\top U \right)^{-1} \right \|_2 \cdot \left\|(RF)^{-1} \right\|_2^2 } = \frac{2}{\beta^2 \cdot \left\|(RF)^{-1} \right\|_2^2 },
\end{equation}
and
\begin{equation}
\label{eq:supL}
\sup_i L_i = \sup_i \frac{2\|A_i F\|_2^2}{p_i} = \sup_i \frac{2\|U_i RF\|_2^2}{p_i} \leq 2 c_1 \|U\|_F^2 \cdot \|RF\|_2^2 = 2c_1 \alpha^2 \cdot \|RF\|_2^2,
\end{equation}
and
\begin{eqnarray}
\label{eq:sigmasq}
\sigma^2 = \ExpectSub{i\sim\D}{\|g_i(y^\ast)\|^2} 
   &=& 4 \sum_{i=1}^n (A_iFy^\ast-b_i)^2 \| A_i F \|^2 / p_i \nonumber \\
   &=& 4 \sum_{i=1}^n (U_iRFy^\ast-b_i)^2 \| U_i RF \|^2 / p_i \nonumber \\
   &\leq& 4 c_1 \|RF\|_2^2 \|U\|_F^2 \left( \sum_{i=1}^n (U_iRFy^\ast-b_i)^2 \right) \nonumber \\
   &=& 4 c_1 \|U\|_F^2 \cdot \|RF\|_2^2 \cdot h(y^\ast) \nonumber \\
   &=& 4c_1 \alpha^2 \cdot \|RF\|_2^2 \cdot h(y^\ast).
\end{eqnarray}

 Equipped with these constant and from Proposition~\ref{prop:sgd_sc}, we have the following error bound of the solution vector estimate $\{ y_t \}_{t=1}^T $ generated by the weighted SGD algorithm
  \begin{eqnarray*}
&&\Expect{\| x_T - x^\ast \|_H^2}  \\
&=&\Expect{ \| y_T - y^\ast \|_2^2 } \nonumber \\ 
&\leq&  \left(1 - 4\eta \sigma_{\min}^2(AF) \left(1-\eta \sup_i \frac{2\|A_i F\|_2^2}{p_i} \right)\right)^T  \|y_0 - y^\ast\|_2^2+ \frac{ 2\eta \sum_{i=1}^n (A_iFy^\ast-b_i)^2 \| A_i F \|_2^2 / p_i}{ \sigma_{\min}^2(AF)(1-\eta \sup_i \frac{2\|A_i F\|_2^2}{p_i})}\nonumber \\
&=&  \left(1 - 4\eta \sigma_{\min}^2(AF) \left(1-\eta \sup_i \frac{2\|A_i F\|_2^2}{p_i} \right)\right)^T  \|x_0 - x^\ast\|_H^2+ \frac{ 2\eta \sum_{i=1}^n (A_i Fy^\ast-b_i)^2 \| A_i F \|_2^2 / p_i}{ \sigma_{\min}^2(AF)(1-\eta \sup_i \frac{2\|A_iF\|_2^2}{p_i})}\nonumber\\
&\leq&\left( \frac{1 - 4\eta\left(1 - 2\eta c_1 \alpha^2\|RF\|_2^2\right)}{\beta^2 \|(RF)^{-1}\|_2^2} \right)^T  \|x_0 - x^\ast\|_H^2 + \frac{ 2 c_1 \eta \bar \kappa_2^2(U) \kappa^2(RF) h(y^\ast)}{1-2\eta c_1\alpha^2\|RF\|_2^2}.
\end{eqnarray*}
Notice that the above equalities follow from the equivalence condition in \eqref{eq:obj_link}.
Combining the results from the above parts completes the proof of this lemma.


\subsection{Proof of Proposition~\ref{cor:l2_new}}
Throughout the proof, we denote
\begin{equation}
f(x) = \|Ax-b\|_2^2, \quad h(y) = \|AFy - b\|_2^2.
\end{equation}
Denote by $ H = \left( F F^\top \right)^{-1} $ the weights of the ellipsoidal norm. Also recall the following constants defined in the statement of proposition
\begin{equation}
c_1 = \frac{1+\gamma}{1-\gamma},\, c_2 = \frac{\|y_0-y^\ast\|_2^2}{\|y^\ast\|_2^2} = \frac{\|x_0 - x^\ast\|_H^2}{\|x^\ast\|_H^2}, \,c_3 = \frac{\|Ax^\ast\|_2^2} { f(x^\ast)}.
\end{equation}
Before diving into the detailed proof, we first show a useful inequality.
   \begin{equation}
    \label{eq:important2}
    c_3 h(y^\ast) = c_3 f(x^\ast)
     =  \|Ax^\ast\|_2^2 = \|URF y^\ast\|_2^2
     \geq \mu \|y^\ast\|_2^2/2.
  \end{equation}

Now we show the first part.
For an arbitrary target error $\epsilon>0$, using \eqref{eq:mu}, \eqref{eq:supL}, \eqref{eq:sigmasq} and setting 
\begin{equation}
 \frac{c_3 \epsilon \cdot h(y^\ast) }{ \|AF\|_2^2 } \rightarrow \epsilon
\end{equation}
in Corollary~\ref{cor:sgd_sc} we have that when the step-size is set to be
\begin{equation}
\eta = \frac{1}{4}\frac{c_3 \epsilon \cdot \sigma_{\min}^2(AF) \cdot h(y^\ast) / \|AF\|_2^2}{ \sum_{i=1}^n (A_iFy^\ast-b_i)^2 \| A_iF \|_2^2 / p_i  + c_3 \left(\epsilon \cdot h(y^\ast) /  \|AF\|_2^2 \right)\sigma_{\min}^2(AF)\sup_i \frac{\|A_i F\|_2^2}{p_i}},
\end{equation}
then after
\begin{eqnarray}
\label{eq:T_1}
&&\log\left( \frac{2\|y_0 - y^\ast\|_2^2}{c_3 \epsilon \cdot h(y^\ast) / ( \|U\|_2^2\|RF\|_2^2)} \right) \left( c_1 \alpha^2 \beta^2 \|RF\|_2^2 \|(RF)^{-1}\|_2^2 + \frac{c_1  \alpha^2\beta^4 \|U\|_2^2  \|RF\|_2^4 \|(RF)^{-1}\|_2^4}{c_3 \epsilon} \right) \nonumber \\
  &\leq&\log\left( \frac{2 \|U\|_2^2\|RF\|_2^2 \cdot \|y_0 - y^\ast\|_2^2}{c_3 \epsilon \cdot h(y^\ast)} \right) \left( c_1 \bar \kappa_2^2(U) \kappa^2(RF) + \frac{c_1  \bar \kappa_2^2(U) \kappa^2(U) \kappa^4(RF)}{c_3 \epsilon} \right) \nonumber \\
  &\leq& \log\left( \frac{2 c_2 \kappa^2(U) \kappa^2(RF) ) }{\epsilon } \right) \left( c_1 \bar \kappa_2^2(U) \kappa^2(RF) \right) \left(1 + \frac{ \kappa^2(U) \kappa^2(RF) }{c_3 \epsilon} \right)
\end{eqnarray}
iterations,
the sequence $\{ y_t \}_{k=1}^T $ generated by running weighted SGD algorithm satisfies the error bound
\begin{equation}
   \|y_T - y^\ast\|_2^2 \leq \frac{c_3 \epsilon \cdot h(y^\ast) }{ \|AF\|_2^2 }.
\end{equation}
Notice that in \eqref{eq:T_1}, we used \eqref{eq:important2}.
From this, we have
\begin{eqnarray}
  \|A(x_T - x^\ast)\|_2^2 &=& \|AFF^{-1}(x_T - x^\ast)\|_2^2 \nonumber \\
   &\leq& \|AF\|_2^2 \cdot \|x_T - x^\ast\|_H^2 \nonumber \\
   &=&  \|AF\|_2^2 \cdot \| y_T - y^\ast \|_2^2 \nonumber \\ 
   &=&  c_3 \epsilon \cdot h(y^\ast) \nonumber \\ 
   &=&  \epsilon \|Ax^\ast\|_2^2.
\end{eqnarray}

For the second part, we show the result for general choice of $F$. The proof is basically the same as that of the first part except that we set
\begin{equation}
 \frac{2\epsilon h(y^\ast) }{ \|AF\|_2^2 } \rightarrow \epsilon
\end{equation}
in Corollary~\ref{cor:sgd_sc}.
The resulting step-size $\eta$ and number of iterations required $T$ become
\begin{equation}
\eta = \frac{1}{4}\frac{2 \epsilon \cdot \sigma_{\min}^2(AF) \cdot h(y^\ast) / \|AF\|_2^2}{ \sum_{i=1}^n (A_iFy^\ast-b_i)^2 \| A_iF \|_2^2 / p_i  + \left(2 \epsilon \cdot h(y^\ast) /  \|AF\|_2^2 \right)\sigma_{\min}^2(AF)\sup_i \frac{\|A_i F\|_2^2}{p_i}}
\end{equation}
and
\begin{equation}
T = \log\left( \frac{c_2 \kappa^2(U) \kappa^2(RF) ) }{\epsilon } \right) \left( c_1 \bar \kappa_2^2(U) \kappa^2(RF) \right) \left(1 + \frac{ \kappa^2(U) \kappa^2(RF) }{2\epsilon} \right).
\end{equation}
Setting $F=R^{-1}$ recovers the value of $T$ shown in Proposition~\ref{cor:l2_new}.
The sequence $\{ y_t \}_{k=1}^T $ generated by running weighted SGD algorithm satisfies the error bound
\begin{equation}
   \|y_T - y^\ast\|_2^2 \leq \frac{ 2\epsilon h(y^\ast) }{ \|AF\|_2^2 }.
\end{equation}
Notice that when the problem is unconstrained, by smoothness of the objective $h(y)$, we have
\begin{equation}
  h(y_T) - h(y^\ast) \leq \|AF\|_2^2 \cdot \|y_T - y^\ast\|_2^2 \leq 2\epsilon h(y^\ast).
\end{equation}
Then by \eqref{eq:obj_link}, we have
\begin{equation}
  f(x_T) \leq (1+2\epsilon) f(x^\ast) \leq (1+2\epsilon+\epsilon^2) f(x^\ast).
\end{equation}
This implies
\begin{equation}
  \sqrt{ f(x_T) } \leq (1+\epsilon) \sqrt{ f(x^\ast)}.
\end{equation}
This completes the proof since $\sqrt{f(x)} = \|Ax-b\|_2$.


\subsection{Proof of Theorem~\ref{prop:coreset}}

 Let $G_f$ consist of $m_f$ copies of $g_f$ and $G = \bigcup_{f \in \F} G_f$.
 We may view the sampling step in Algorithm~\ref{alg:coreset} as follows.
 Sample $s$ items uniformly from $G$ independently with replacement and denote the corresponding subset of samples by $S$.
 Then rescale every function in $S$ by $M(\F)/s$ and obtain $\D$.
 
 By Theorem~4.1 in~\citet{Feldman_coreset}, we know that if the above intermediate set $S$ is an $(\epsilon \cdot n/M(\F))-$approximation of the set $G$, then the resulting set $\D$ is a desired $\epsilon$-coreset for $\F$.
 Indeed, $S$ is such a set according to Theorem~6.10 in~\citet{Feldman_coreset}.


\subsection{Proof of Proposition~\ref{prop:sensitivity}}

We use $A$ to denote $\bar A$ for the sake of simplicity.
Also define the sensitivity at row index $i\in[n]$ as
\begin{equation}
 \label{eq:sigma_i}
  s_i = n \cdot \sup_{x \in \C} \frac{ |A_i x|^p }{ \sum_{j=1}^n |A_j x|^p}.
\end{equation}
Suppose $U \in \reals^{n\times k}$ is an $(\alpha, \beta)$ well-conditioned basis of the range space of $A$
satisfying $A = UR$, where $k = \textrm{rank}(A)$ and $R \in \reals^{k \times (d+1)}$.
Then from \eqref{eq:sigma_i}, we have that
\begin{equation}
 \frac{s_i}{n} = \sup_{x \in \C} \frac{  \lvert A_i x  \rvert^p} { \|Ax\|_p^p }  
     = \sup_{x \in \C} \frac{  \lvert U_iR x  \rvert^p }{ \|URx\|_p^p }  
     = \sup_{y \in \C'} \frac{  \lvert U_i y  \rvert^p }{ \|Uy\|_p^p }  
     \leq \sup_{y \in \C'} \frac{ \| U_i\|_p^p \|y\|_q^p }{ \|y\|_q^p /\beta^p } 
     = \beta^p \| U_i \|_p^p  
     = \beta^p \cdot \lambda_i , 
\end{equation}
where $\C' = \{ y \in \reals^d \vert y = Rx, x \in \C \} $ is a one-to-one mapping. The first inequality follows from H\"{o}lder's inequality with $\frac{1}{p}+ \frac{1}{q} = 1$ and the properties of well-conditioned bases.
According to the definition of sensitivity $m(f_i) = \lfloor s_i \rfloor + 1$, the above property implies
\begin{equation}
   m(f_i) \leq n\beta^p \lambda_i + 1.
\end{equation}
which implies
$M(\F) = \sum_{i=1}^n s_i \leq (n \beta^p \sum_{i=1}^n \lambda_i) + n = n((\alpha \beta)^p+1)$,
and completes the proof.


\subsection{Proof of Proposition~\ref{prop:dim}}

 According to Definition~\ref{def:dimension}, we only have to show that for any arbitrary constant $n$ and set of points $G=\{a_1, \ldots, a_n\} \subseteq \reals^d$, the following condition holds:
  $$ \lvert \{ \textup{\bf Range}(G,x,r) \vert x \in \X, r \geq 0 \} \rvert \leq n^{d+1}, $$ 
 where $\textup{\bf Range}(G,x,r) = \{ a_i \vert \lvert a_i^\top  x \rvert^p \leq r \}$ is the region located in the $p-$norm ellipsoid $\lvert a_i^\top  x \rvert^p = r $.
 Since the following  condition holds: $\{ a_i \vert \lvert a_i^\top  x \rvert^p \leq r \} = \{ a_i \vert \lvert a_i^\top  x \rvert \leq r^\frac{1}{p} \}$ and the constant $r$ is non-negative and arbitrary. Without loss of generality, we assume $p=1$ in the above definition, i.e., $\textup{\bf Range}(G,x,r) = \{ a_i \vert \lvert a_i^\top  x \rvert \leq r \}$.

Notice that for every $x$ and $r$,
 $\textup{\bf Range}(G,x,r)$ is a subset of $G$.
 Hence, we may view it as a binary classifier on $G$, denoted by $c_{x,r}$.
 Given $x\in \X$ and $r \geq 0$,
 for any $a_i \in G$ we have that
 \begin{equation*}
  c_{x,r}(a_i) =
  \begin{cases}
     1, & \mbox{if } |a_i^\top  x| \leq r; \\
     0, & \mbox{otherwise}.
  \end{cases}
 \end{equation*}
 Therefore, one immediately sees that $\lvert \{ \textup{\bf Range}(G,x,r) \vert x \in \X, r \geq 0 \} \rvert$ is the shattering coefficient of $C := \{ c_{x,r} \vert x \in \X, r \geq 0 \}$ on $n$ points, denoted by $s(C,n)$.
 To bound the shattering coefficient of $C$, we provide an upper bound based on its VC dimension.
 
 We claim that the VC dimension of $C$ is at most $d+1$. By contradiction, suppose
 there exists $n+2$ points such that any labeling on these $n+2$ points can be shattered by $C$.
 By Radon's Theorem~\citep{radon}, we can partition these points into two disjoint subsets, namely, $V$ and $W$ with size $n_1$ and $n_2$ respectively, where the intersection of their convex hulls is nonempty.
 Let $b$ be a point located in the intersection of the convex hulls of $V$ and $W$, which in general can be written as 
 \begin{equation}
  \label{eq:b_convex}
  b = \sum_{i=1}^{n_1} \lambda_i v_i = \sum_{i=1}^{n_2} \sigma_i w_i,
 \end{equation}
 where $\lambda_i \geq 0$, $\sigma_i \geq 0$ and $\sum_{i=1}^{n_1} \lambda_i = \sum_{i=1}^{n_2} \sigma_i = 1$. 
 
By the above assumption, we can find vector $x\in\reals^n$ and nonnegative constant $r$ such that the following conditions hold:
 \begin{eqnarray}
    -r \leq x^\top  v_i \leq r, ~~~ i = 1,\ldots,n_1; \label{eq:v}  \\
     x^\top  w_i > r \mbox{ or } x^\top  w_i < -r, ~~~ i = 1,\ldots,n_2. \label{eq:w}
 \end{eqnarray}
By combining the conditions in \eqref{eq:b_convex}, \eqref{eq:v} and \eqref{eq:w},
we further obtain both inequalities
 \begin{equation}
 -r \leq b^\top  x \leq r,
 \end{equation}
 and
   \begin{equation}
    b^\top  x < -r \mbox{ or } b^\top  x > r,
   \end{equation}
  which is clearly paradoxical! This concludes that the VC dimension of $C$ is less than or equal to $d+1$.
Furthermore, by Sauer's Lemma~\citep{sauer},
for $n \geq 2$ the shattering coefficient $s(C, n)=\lvert \{ \textup{\bf Range}(G,x,r) \vert x \in \X, r \geq 0 \} \rvert$ is less than $n^{d+1}$, which completes the proof of this proposition.


\subsection{Proof of Proposition~\ref{prop:counter_example}}
 Without loss of generality, assume the low dimension $d$ is even (because if $d$ is odd, we can always add an extra arbitrary row to input matrix $A$ and upper bound the size of the original total sensitivity set by the same analysis).
 Let $a_i \in [0,1]^d$ be a vector with exactly $d/2$ elements to be $1$.
 For each $i \in [n]$, let $B_i = \{j \vert a_{ij} = 1\}$, where $a_{ij}$ denotes the $j$-th element of vector $a_i$.
 For fixed $i$,
 define $x$ as follows,
  \begin{equation}
   x_j =  \begin{cases}
        2/d, & \mbox{if~~~} j \in B_i, \\
        -d, & \mbox{otherwise}.
  \end{cases}
  \end{equation}
One immediately notices from the above expression that $x^\top  a_i = 1$.
 Thus for $j \neq i$, $a_j \neq a_i$, there exists an index $k \in [d]$ such that $a_{jk} = 1$ but $a_{ik} = 0$.
  Furthermore the above condition implies 
  \begin{equation}
   x^\top  a_j = \sum_{l=1}^d x_l a_{jl} = \sum_{l \in B_j, l \neq k}^d x_l a_{jl} + \sum_{l \neq B_j}^d x_l a_{jl} + x_k a_{jk} \leq (d/2-1)(2/d) - d <0,
   \end{equation}
  which further implies $f_j(x) = x^\top  a_j = 0$;
  Therefore, the $i$-th sensitivity becomes
   \begin{equation}
    s_i = \sup_x \frac{f_i(x)}{\sum_{i=j }^n f_j(x)} \geq 1. 
   \end{equation}
  Since the above condition holds for arbitrary index $i \in [n]$, and we have $d \choose d/2$ number of vectors $a_i$, i.e., $n = {d \choose d/2}$, this concludes that 
  the size of the total sensitivity set is at least ${d \choose d/2} \approx 2^d$.

\section{Stochastic Gradient Descent}
\label{sec:sgd}

Consider minimizing the following objective
\begin{equation}
 \label{eq:sgd_prob}
  \text{minimize}_{x \in \X} f(x) = \ExpectSub{i \sim P}{ f_i(x) }.
\end{equation}
Stochastic gradient descent (SGD) exploits the following update rule
\begin{equation}
 \label{eq:sgd_update}
  x_{t+1} = \arg\min_{x \in \X} \eta \langle x-x_t, g_{\xi_t} (x_t) \rangle + \frac{1}{2} \|x-x_t\|_2^2,
\end{equation}
where $\xi_t \in [n]$ is an index drawn according to $P$, $g_{\xi_t}(x) = \nabla f_{\xi_t}(x)$ and $\ExpectSub{\xi_t \sim P}{f_{\xi_t}(x)} = f(x)$. When $\X = \reals^d$, the update rule~\eqref{eq:sgd_update} boils down to $x_{t+1} = x_t - \eta  g_{\xi_t} (x_t)$. 
Note here, if $f_{\xi_t}(x)$ is not differentiable, we take  $g_{\xi_t}(x_t)$ to be one of its sub-gradients, i.e., $g_{\xi_t}(x_t) \in \partial f_{\xi_t} (x_t)$. In this case, SGD boils down to stochastic sub-gradient method. For simplicity, we still refer to the algorithms as SGD.

In the following, we present two results regarding the convergence rate of SGD on problem with non-strongly convex objective and strongly convex objective, respectively.

\subsection{Non-strongly convex case}
Here we analyze the case where the objective function $f(x)$ is not strongly convex. Also, each sub-function is not necessary differentiable. That is, $g_i(x)$ can be a sub-gradient of function $f_i$ at $x$.

\begin{proposition}\label{prop:expe}
Assume that $\frac{1}{2} \| \cdot \|_2^2 \geq \frac{\lambda}{2} \| \cdot \|^2$ for some norm $\|\cdot\|^2$.
Also assume that $\|g_t(x_t)\|_\ast \leq M$ for any $t>0$ where $\| \cdot \|_\ast$ is the dual norm of $\| \cdot \|$.
The output $\bar x = \frac{1}{T+1} \sum_{t=1}^T x_t$ of SGD satisfies, for any $y \in \X$,
\begin{equation}
 \Expect{f(\bar x) } - f(y)  \leq \frac{\|y - x_0\|_2^2}{2\eta (T+1)} + \frac{\eta}{2\lambda} M^2.
\end{equation}
In particular, when $\eta = \frac{\|y - x_0\|_2}{M} \sqrt{\frac{\lambda}{T+1}}$,
we have
\begin{equation}
 \Expect{f(\bar x) } - f(y) 
 \leq M \|y - x_0\|_2 \sqrt{\frac{1}{(T+1)\lambda} }.
\end{equation}
\end{proposition}
\begin{proof}
From Lemma 1 in \citet{composite_mirror}, at step $t$, we have that
\begin{equation}
 \label{eq:comid_basic}
 \eta(f_t(x_t) - f_t(y) ) \leq \frac{1}{2} \|y - x_t\|_2^2 - \frac{1}{2} \| y-x_{t+1}\|_2^2 + \frac{\eta^2}{2\lambda}\|g_t(x_t)\|_\ast^2. 
\end{equation}
Conditioned on $x_t$, taking the conditional expectation with respect to $\xi_t$ on both sides, we have
\begin{equation}
 \Expect{\eta(f_t(x_t) - f_t(y) ) \vert x_t} \leq \Expect{ \frac{1}{2} \|y-x_t\|_2^2 - \frac{1}{2} \|y- x_{t+1}\|_2^2 + \frac{\eta^2}{2\lambda}\|g_t(x_t)\|_\ast^2 \vert x_t}.
\end{equation}
Noticing that $\ExpectSub{\xi_t \sim P}{f_t(x)} = f(x)$, we have
\begin{equation}
 \eta f(x_t) - \eta f(y) \leq \frac{1}{2} \|y-x_t\|_2^2 + \Expect{ - \frac{1}{2} \|y-x_{t+1}\|_2^2 + \frac{\eta^2}{2\lambda}\|g_t(x_t)\|_\ast^2 \vert x_t}.
\end{equation}
Then by taking the expectation over $x_t$ and using the fact that $\|g_t(x_t)\|_\ast \leq M$, we have
\begin{equation}
 \Expect{ \eta f(x_t)} - \eta f(y)  \leq \Expect{ \frac{1}{2} \|y-x_t\|_2^2} - \Expect{  \frac{1}{2} \|y-x_{t+1}\|_2^2} + \frac{\eta^2}{2\lambda} M^2.
\end{equation}
Summing up the above equation with $t=0,\ldots,T$ and noticing $\|y - x_{t+1}\|_2^2 \geq 0$, we have
\begin{equation}
   \eta \sum_{t=0}^T  \Expect{ f(x_t) } - \eta (T+1) f(y)  \leq \frac{1}{2} \|y-x_0\|_2^2 + \frac{\eta^2 (T+1)}{2\lambda} M^2  .
\end{equation}
Finally by convexity of $f$, we have that
\begin{equation}
 \Expect{f(\bar x)} - f(y)  \leq \frac{ \|y-x_0\|_2^2}{2\eta (T+1)} + 
\frac{\eta}{2\lambda} M^2.
\end{equation}
In particular with $\eta = \frac{\|y-y_1\|_2}{M} \sqrt{\frac{\lambda}{T+1}}$,
we have
\begin{equation}
 \Expect{f(\bar x) } - f(y) 
 \leq M \|y-x_0\|_2 \sqrt{\frac{1}{(T+1)\lambda}},
\end{equation}
which completes the proof.

\end{proof}

\subsection{Strongly convex case}
Here we analyze the case where the objective function $f(x)$ is strongly convex.
We make the following two assumptions:
\begin{itemize}
\item [(A1)]
Function $f(x)$ is strongly convex with modulus $\mu$. That is, for any $x,y \in \X$,
\begin{equation}
  f(y) \geq f(x) + \langle \nabla f(x), y-x \rangle + \frac{\mu}{2} \|y-x\|_2^2.
\end{equation}
\item [(A2)]
For each $i\in [n]$, the gradient of each sub-function $\nabla f_i(x)$ is Lipschitz continuous with constant $L_i$. That is, for any $x,y \in \X$,
\begin{equation}
  \| \nabla f_i(y) - \nabla f_i(x) \|_2 \leq L_i \|y-x\|_2.
\end{equation}
\end{itemize}
The following results also appeared in~\cite{needel-weightedsgd}.

\begin{proposition}
\label{prop:sgd_sc}
Under assumption (A1), (A2), the sequence $\{x_t\}$ generated by SGD satisfies
\begin{equation}
\label{eq:sgd_sc_bound}
\Expect{ \| x_T - x^\ast \|_2^2 } \leq (1 - 2\eta \mu(1-\eta \sup L_i))^T  \|x_0 - x^\ast \|_2^2 + \frac{\eta \sigma^2}{\mu(1-\eta \sup L_i)},
\end{equation}
where $\sigma^2 = \ExpectSub{i\sim\D}{\|\nabla f_i(x^\ast)\|_2^2}$ and $x^\ast$ is the optimal solution to~\eqref{eq:sgd_prob}.
\end{proposition}
\begin{proof}
The proof essentially follows the same lines of arguments as in~\cite{needel-weightedsgd}.
The only difference is that, here we are working on the constrained problem where update rule~\eqref{eq:sgd_update} is equivalent to
\begin{equation}
  x_{t+1} = \Pi_{\X} (x_t - \eta g_t(x_t)).
\end{equation}
Notice that $\Pi_{\X}(x)$ is a projection operator to the feasible set $\X$ and it is non-expansive. This further implies
\begin{equation}
 \| x_{t+1} - x^\ast \|_2^2 = \| \Pi_{\X} (x_t - \eta g_t(x_t)) - x^\ast \|_2^2 \leq \| x_t - \eta g_t(x_t) - x^\ast \|_2^2.
\end{equation}
The rest of the proof follows analogous arguments in~\cite{needel-weightedsgd}.
\end{proof}

\begin{corollary}
\label{cor:sgd_sc}
Given a target accuracy $\epsilon>0$, and let the step-size be $\eta = \frac{\epsilon \mu}{2\sigma^2 + 2\epsilon \mu \sup L_i}$. Then after
\begin{equation}
  T \geq \log\left( \frac{2\|x_0 - x^\ast\|^2}{\epsilon} \right) \left( \frac{\sigma^2}{\epsilon \mu^2} + \frac{\sup L_i}{\mu} \right)
\end{equation}
iterations, we have that
 \begin{equation}
   \Expect{\|x_T - x^\ast\|_2^2} \leq \epsilon.
 \end{equation}
\end{corollary}
\begin{proof}
The proof can be found in~\cite{needel-weightedsgd}.
\end{proof}

\end{document}